\documentclass[11pt,a4paper,oneside]{amsart}
\usepackage{amssymb,amsmath,mathrsfs,amsthm,amsrefs,graphicx,color}
\usepackage{leftidx}

\theoremstyle{plain}
\newtheorem{theorem}{Theorem}[section]

\newtheorem{lemma}[theorem]{Lemma}
\newtheorem{corollary}[theorem]{Corollary}
\newtheorem{proposition}[theorem]{Proposition}
\newtheorem{assumption}[theorem]{Assumption}
\theoremstyle{remark}
\newtheorem{remark}[theorem]{Remark}

\allowdisplaybreaks[4]

\usepackage{hyperref}

\numberwithin{equation}{section}

\newcommand{\C}{\mathbb{C}}
\newcommand{\R}{\mathbb{R}}

\newcommand{\Q}{\mathbb{Q}}
\newcommand{\F}{\mathcal{F}}

\renewcommand{\Im}{\operatorname{Im}}
\renewcommand{\Re}{\operatorname{Re}}

\newcommand{\norm}[1]{\left\lVert #1\right\rVert}

\newcommand{\tnorm}[1]{\lVert #1\rVert}

\def\({\left(}
\def\){\right)}
\def\<{\left\langle}
\def\>{\right\rangle}
\def\le{\leqslant}
\def\ge{\geqslant}

\def \F{\mathcal{F}}

\newcommand{\eps}{\varepsilon}

\DeclareMathOperator{\sign}{sign}
\DeclareMathOperator{\sn}{sn}
\DeclareMathOperator{\cn}{cn}
\DeclareMathOperator{\dn}{dn}
\DeclareMathOperator{\cd}{cd}
\DeclareMathOperator{\sd}{sd}
\DeclareMathOperator{\nd}{nd}

\DeclareMathOperator{\pq}{pq}

\DeclareMathOperator{\rank}{rank}
\DeclareMathOperator{\Ker}{Ker}
\DeclareMathOperator{\tr}{tr}
\DeclareMathOperator{\Span}{Span}

\newcommand{\todayd}{\the\year/\the\month/\the\day}

\theoremstyle{definition}

\newcommand{\ol}{\overline}

\begin{document}
\title[Cubic NLS Systems without Coercive Conserved Quantity]{
Global existence and large-time behavior\\ of solutions to cubic nonlinear Schr\"odinger systems\\ without coercive conserved quantity
}

\author[S. Masaki]{Satoshi MASAKI}
\address{Department of mathematics, 
Hokkaido University, Sapporo Hokkaido, 060-0810, Japan}
\email{masaki@math.sci.hokudai.ac.jp}
%
%

\begin{abstract}
In this article, we investigate the large-time behavior of small solutions to a system of one-dimensional cubic nonlinear Schr\"odinger equations with two components.
In previous studies, a structural condition on the nonlinearity has been employed to guarantee the existence of a coercive, mass-type conserved quantity.
We identify a new class of systems that do not satisfy such a condition and thus lack a coercive conserved quantity. Nonetheless, we establish the global existence and describe the large-time behavior for small solutions in this class.
In this setting, the asymptotic profile is described in terms of solutions to the corresponding system of ordinary differential equations (ODEs).
A key element of our analysis is the use of a quartic conserved quantity associated with the ODE system.
Moreover, for a specific example within this class, we solve the ODE system explicitly, showing that the asymptotic behavior is expressed using Jacobi elliptic functions.

\end{abstract}

\maketitle

\section{Introduction}

In this article, we consider the initial value problem of a system of cubic NLS equations composed of two components
\begin{equation}\label{E:generalNLS}
	\left\{
	\begin{aligned}
	&(i  \partial_t + \partial_x^2 )u_j = F_j(u_1,u_2),  \quad (t,x) \in \mathbb{R}\times \mathbb{R}, \quad j=1,2,\\
	&(u_1(0),u_2(0)) = (u_{0,1},u_{0,2}) \in H^{0,1} \times H^{0,1}
	\end{aligned}
	\right.
\end{equation}
with
\begin{equation}\label{E:nonlinearity}
	\left\{
	\begin{aligned}
	F_1(u_1,u_2) ={}& \lambda_1 |u_1|^2 u_1 + \lambda_2 |u_1|^2 u_2 + \lambda_3 u_1^2 \overline{u_2} +
	\lambda_4 |u_2|^2 u_1 + \lambda_5 u_2^2 \overline{u_1} + \lambda_6 |u_2|^2 u_2,\\
	F_2(u_1,u_2)={}& \lambda_7 |u_1|^2 u_1 + \lambda_8 |u_1|^2 u_2 + \lambda_9 u_1^2 \overline{u_2} +
	\lambda_{10} |u_2|^2 u_1 + \lambda_{11} u_2^2 \overline{u_1} + \lambda_{12} |u_2|^2 u_2,
	\end{aligned}
	\right.
\end{equation}
where $(\lambda_1, \dots, \lambda_{12}) \in \mathbb{R}^{12}$.
We aim to establish the global existence of solutions for small initial data and study their asymptotic behavior under a suitable condition on the parameters.

It is well known from previous studies that the cubic nonlinear term is critical in one dimension. Firstly, let us briefly review for a single equation. 
In the single-equation case,
the gauge-invariant cubic NLS has the form:
\begin{equation}\label{E:csNLS}
	(i  \partial_t + \partial_x^2 )u = \lambda |u|^2 u,
\end{equation}
with a number $\lambda \in \mathbb{C}$.
The number $\lambda$ indicates the property of the system.
If it is a real number $\lambda \in \mathbb{R}$, the modified scattering-type asymptotic behavior of small solutions has been obtained by Ozawa \cite{Oz}:
\begin{equation}\label{E:mscattering}
	u(t,x) = (2it)^{-\frac12} e^{i\frac{x^2}{4t}} \widehat{u_+}(\tfrac{x}{2t}) \exp \( -i \lambda |\widehat{u_+}(\tfrac{x}{2t})|^2 \tfrac{\log t}{2}  \) + O(t^{-\frac34+})
\end{equation}
in $L^\infty$  as $t\to\infty$,
where $\widehat{u_+}$ is a suitable function (see also \cite{GO,HN,KP,IT,MM,MMU,Mu}).
When $\lambda$ is not real, the NLS equation exhibits dissipative or amplification effects \cite{Shimomura,Kita}. It is noteworthy that in these cases, the global existence of solutions  is not trivial. In particular, systems with amplification effects are known to possess a blowup solution (see \cite{Kita}).
We refer the readers to \cite{MM2,MurPus} and references therein for studies in this direction of a wider class of single equations.

Even for the cubic NLS systems
of the form \eqref{E:generalNLS}, i.e., for a system that consists with two components and has real coefficients for the nonlinear terms, the system exhibits more diversity than the single equation case. Indeed, an interesting example of the system of the form \eqref{E:generalNLS} is given in
 \cite[Appendix B]{MSU2}.
The system
\begin{equation}
	\left\{
	\begin{aligned}
	& (i  \partial_t + \partial_x^2 )u_1 = -2(|u_1|^2- |u_2|^2) u_1 + \overline{u_1}u_2^2,  \\
	& (i  \partial_t + \partial_x^2 )u_2 = -2(|u_1|^2- |u_2|^2) u_2 - u_1^2\overline{u_2} 
	\end{aligned}
	\right.
\end{equation}
admits at least three types of solutions. The first type, represented by solutions of the form $(u,0)$ and $(0,u)$, exhibits the modified scattering behavior described in \eqref{E:mscattering}.
This is because $u$ in these forms solves \eqref{E:csNLS} with $\lambda=-1$ and $\lambda=1$, respectively.
The second type, $(u,e^{-i\pi/4} u)$, exhibits a dissipative effect, while the third type, $(u,e^{i\pi/4} u)$, demonstrates amplification effects for positive time direction.
In these ansatzes, $u$ solves \eqref{E:csNLS} with $\lambda=-i$ and $\lambda=i$, respectively.
The example also highlights that the global existence of solutions  is not always guaranteed 
even when the coefficients are real, as in  \eqref{E:generalNLS}.

A simple yet powerful tool for ensuring global existence is the use of conservation laws, which are intrinsic properties of the system.
Specifically, if a system possesses a conserved quantity of the form
\begin{equation}\label{E:quadconserve}
\int_\mathbb{R} (a |u_1|^2 + 2b\Re (\overline{u_1}u_2) + c |u_2|^2)dx, \quad ac > b^2,
\end{equation}
then solutions exist globally in time for arbitrary initial data in $L^2 \times L^2$.
The condition $ac > b^2$ ensures the coercivity of this conserved quantity, i.e., its equivalence to the square of the $L^2 \times L^2$-norm.
The existence of such a conserved quantity, as well as its coercivity, is determined by the structure of the system's nonlinearity.

In this article, we refer to this condition as the \emph{weak null gauge condition}, as it represents a weaker form of the null gauge condition introduced in \cite{Tsu,KaTsu}, which was originally motivated by studies of nonlinear wave equations (see \cite{Ka}).

The results concerning the asymptotic behavior of system \eqref{E:generalNLS} for those satisfying this weak null condition are available in \cite{NST,KS,M3,U}. Particularly, in \cite{KS}, it is shown that small solutions have the following asymptotic behavior:
\begin{equation}\label{E:KSasymptotics}
	u_j(t,x) = (2it)^{-\frac12} e^{i\frac{x^2}{4t}} A_j \( \tfrac{\log t}{2} ; \tfrac{x}{2t} \) + O(t^{-\frac34+})
\end{equation}
in $L^\infty$ as $t\to\infty$,
where $(A_1,A_2)$ are one-parameter families of solutions to the corresponding system of
ordinary differential equations (ODE system):
\begin{equation}\label{E:Asys}
	i \tfrac{d}{dt} A_j	= 	 F_j(A_1,A_2) , \quad j=1,2.
\end{equation}
We remark that \eqref{E:Asys} is 
obtained by removing $\partial_x^2 u_j$ from \eqref{E:generalNLS}. (In \cite{KS}, systems of a wider class than \eqref{E:generalNLS}, including derivative nonlinearity, are treated. See also  \cite{HNS,KN,NST}.) 
It is notable that the asymptotic behavior \eqref{E:KSasymptotics} can be understood as a natural extension of the modified scattering-type asymptotic behavior \eqref{E:mscattering} to the systems.
Indeed, the leading term of the asymptotic behavior in \eqref{E:mscattering} for the single case is represented as 
\[
	 (2it)^{-\frac12} e^{i\frac{x^2}{4t}} A \( \tfrac{\log t}{2} ; \tfrac{x}{2t} \)
\]
with a one-parameter family $A$ of solutions  to the ordinary differential equation:
\[
	i\tfrac{d}{dt} A=\lambda |A|^2 A, \quad A(0;\xi) = \widehat{u_+}(\xi).
\]
It can be verified from the fact that the solution is given as $A(\tau;\xi)=\widehat{u_+}(\xi) \exp (-i\lambda |\widehat{u_+}(\xi)|^2 \tau)$.
The ordinary differential equation is obtained by dropping $\partial_x^2$ from \eqref{E:csNLS}.

Furthermore, for systems where the quantity \eqref{E:quadconserve} is monotonic with respect to time, global existence (in one direction of time) can be established in a similar manner.
To distinguish this condition from the weak null gauge condition discussed above, we refer to it as a \emph{dissipative condition}\footnote{This condition is sometimes also referred to as a "weak null condition" (see, e.g., \cite{KaSu}).} in this article.

As mentioned above, the single-equation case has been studied in \cite{HNS,Shimomura}.
In \cite{LS,LS2,SaSu}, the decay of solutions for dissipative systems was obtained.
More recently, the asymptotic behavior and related decay properties of solutions to dissipative NLS equations and systems have been extensively studied (see, e.g., \cite{LNSS1,LNSS2,LNSS5,KiSa,KiSa2,Sa} and references therein).

In this article, we treat systems that do not possess coercive conserved quantities and are not dissipative either. We will specify the conditions for the coefficients later. The simplest model case that satisfies our condition is given by:
\begin{equation}\label{E:NLS1}
	\left\{
	\begin{aligned}
	& (i  \partial_t + \partial_x^2 )u_1 =  |u_2|^2 u_2,  \\
	& (i  \partial_t + \partial_x^2 )u_2 =  |u_1|^2 u_1 .
	\end{aligned}
	\right.
\end{equation}
This system possesses 
\[
	\int_\mathbb{R} \Re \overline{u_1} u_2 \, dx
\]
and
\[
	\int_\mathbb{R} \Re \overline{\partial_x u_1} \partial_x u_2 \, dx + \frac14 \int_\mathbb{R} 
	(|u_1|^4 + |u_2|^4) \, dx
\]
as the conserved quantities.
However, since the quadratic part of these quantities are not coercive, these do not imply the global existence of solutions. 
The corresponding ODE system for this system is given by:
\begin{equation}\label{E:ODE1}
	\left\{
	\begin{aligned}
	&i \tfrac{d}{dt} A_1  = |A_2|^2 A_2,\\
	&i \tfrac{d}{dt} A_2  = |A_1|^2 A_1.
	\end{aligned}
	\right.
\end{equation}
The crucial point is that the solution of this ODE system preserves a \emph{quartic quantity}:
\[
	|A_1|^4 + |A_2|^4.
\]
By utilizing this quantity, we can establish the global existence of small solutions to \eqref{E:NLS1} and prove the asymptotics as in \eqref{E:KSasymptotics}. The systems under consideration are those whose corresponding ODE systems possess an effective quartic conserved quantity. 


In the above example, the conserved quartic quantity for \eqref{E:ODE1} coincides with the integrand of the second term of the conserved “energy-like” quantity for \eqref{E:NLS1}. However, this is merely a coincidence (see Remark \ref{R:coincidence}, below). We also note that this quartic quantity is conserved only for the corresponding ODE system, and its spatial integral is not a conserved quantity for the NLS system. This can be seen, for instance, from the fact that the $L^4\times L^4$-norm of the solutions studied in this article decays in time.

%
The result presented in this article is part of a program aimed at comprehensively understanding the complexity of nonlinear structures in systems of the form \eqref{E:generalNLS}. A classification of general cubic systems of the form \eqref{E:generalNLS} was provided in \cite{MSU2,M}, where all standard forms for systems with mass-type conserved quantities were identified (see also \cite{M2,M3}).

The asymptotic behavior of systems that neither satisfy the weak null gauge condition nor exhibit a dissipative structure has been investigated in \cite{MSU2,KMSU} (see also \cite{KMS,MSU1} for related studies on wave and Klein-Gordon systems). In these works, the authors studied systems consisting of one closed single equation and one linear equation with a time-dependent potential determined by the solution of the single equation.
In contrast, the systems of the form \eqref{E:generalNLS} considered in this article lack such a structure.

\subsection{Main results}
Let $H^{0,1}:= \{ f\in L^2(\R) \ |\ x f\in L^2\}$ be the weighted $L^2$ space with the norm $\|f\|_{H^{0,1}}^2= \|f\|_{L^2}^2 + \|xf\|_{L^2}^2$.
Let $M_3(\R)$ be the set of square matrices of size $3$.

To describe the class of systems which we handle, we recall the matrix-vector representation of a system \eqref{E:generalNLS} introduced in \cite{MSU2}.
The system \eqref{E:generalNLS} can be identified with a pair of a matrix $\mathscr{A} \in M_3 (\R) \simeq \R^9$ and a vector $\mathscr{V} \in \R^3$ as follows:
Given $(\lambda_1, \dots, \lambda_{12}) \in \R^{12}$, we define
\begin{equation}\label{E:matrixC}
\mathscr{A}:=
\begin{bmatrix}
\lambda_2-\lambda_3 & -\lambda_1+\lambda_8-\lambda_9 & -\lambda_7 \\
\lambda_5 & -\lambda_3+\lambda_{11} & -\lambda_9 \\
\lambda_6 & -\lambda_4+\lambda_5+\lambda_{12} & -\lambda_{10}+\lambda_{11}
\end{bmatrix}  
\end{equation}
and
\begin{equation}\label{E:vectorV}
	\mathscr{V}:=
	\begin{bmatrix} \lambda_8-2\lambda_9 \\ \tfrac12 (-\lambda_2+2\lambda_3-\lambda_{10}+ 2\lambda_{11}) \\  \lambda_4-2\lambda_5 \end{bmatrix}. 
\end{equation}

For a square matrix $\mathscr{G}\in M_3(\R)$ of size $3$ and a number $\lambda$, we let $W(\lambda,\mathscr{G})$ be the eigenspace of $\mathscr{G}$ associated with $\lambda$ given by
\[
	W(\lambda, \mathscr{G}) :=\Ker(\mathscr{G}-\lambda I) = \{ {\bf v}\in \R^3 ; \mathscr{G}{\bf v} = \lambda {\bf v} \} \subset \R^3,
\]
where $I$ is the identity matrix.
The eigenspace
$W(\lambda, \mathscr{G})$ is a non-trivial subspace of $\R^3$
if and only if $\lambda$ is an eigenvalue of $\mathscr{G}$.
We further let
\begin{align*}
\mathcal{P}_+&:=\{(a,b,c) \in \R^3 ; ac-b^2>0\}.
\end{align*}
Note that if $(a,b,c) \in \mathcal{P}_+$ then the quadratic form $ax^2+2bxy+cy^2$ is sign-definite.

Our assumption on the system is phrased as follows:
\begin{assumption}\label{A:main}
Suppose that there exists $k>0$ such that
$W(-k^2, \mathscr{A}^2)\cap \mathcal{P}_+\neq \emptyset$.
\end{assumption}
\begin{remark}
The assumption 
implies that $\mathscr{A}$ itself has pure imaginary eigenvalue $\pm i k$.
Since $\mathscr{A}$ has real entries, if $\mathscr{A}^2$ has a negative eigenvalue $-k^2$ then 
$W(-k^2,\mathscr{A}^2)$ is a two-dimensional subspace and hence it is  denoted as $W(-k^2,\mathscr{A}^2)={\bf v}^\perp$ with a non-zero vector ${\bf v} =\ltrans{(v_1,v_2,v_3)}$. 
In this notation,
the property $W(-k^2,\mathscr{A}^2) \cap \mathcal{P}_+\neq \emptyset$ is 
characterized as $v_2^2-4v_1v_3>0$
(see \cite[Theorem 3.1]{MSU1}).
\end{remark}
\begin{remark}
Let us confirm that \eqref{E:NLS1} satisfies the above assumption.
The matrix-vector representation of \eqref{E:NLS1} is
\begin{equation}\label{E:modelAV}
\mathscr{A}=
\begin{bmatrix}
0 & 0 & -1 \\
0 & 0 & 0 \\
1 & 0 & 0
\end{bmatrix}  
,\quad
	\mathscr{V}=
	\begin{bmatrix} 0 \\ 0 \\  0 \end{bmatrix}. 
\end{equation}
One sees that $\ltrans{(1,0,1)} \in W(-1,\mathscr{A}^2) \cap \mathcal{P}_+$ and hence that \eqref{E:NLS1} satisfies Assumption \ref{A:main}.
\end{remark}

The main theorem is as follows. Firstly, it establishes the global existence of solutions and the asymptotic behavior of solutions for general systems. The asymptotic behavior is described using a family of solutions to the corresponding ODE system \eqref{E:Asys}.

\begin{theorem}\label{T:main}
There exists $\delta>0$ and
$\eps_0>0$ such that if $\eps:= \norm{u_{0,1}}_{H^{0,1}} + \norm{u_{0,2}}_{H^{0,1}}$ satisfies $\eps \le \eps_0$ then there exists a unique global solution $(u_1,u_2) \in C(\R;L^2 \times L^2)$ such that $(U(-t)u_1,U(-t)u_2) \in C(\R;H^{0,1} \times H^{0,1})$.
The solution satisfies
\[
	\sup_{t\in \R} (\|u_j(t)\|_{L^2}+|t|^\frac12 \|u_j(t)\|_{L^\infty})
	\lesssim \eps
\]
for $j=1,2$.
Further, there exists $(\psi^\pm_1,\psi_2^\pm) \in (C\cap L^\infty)(\R)$
such that $\tnorm{\psi^\pm_j}_{L^\infty} \lesssim \eps$ and
\begin{equation}\label{E:mainap}
	u_j(t,x) = (2it)^{-\frac12} e^{i \frac{x^2}{4t}} A_j^\pm \(\tfrac{t}{2|t|} \log |t|, \tfrac{x}{2t}\)
	+O(\eps |t|^{-\frac{3}4 +\frac1{2p}+ \delta \eps^2})
\end{equation}
in $L^p_x(\R)$ as $t\to \pm\infty$ for $p \in [2,\infty]$,
where $(A_1^\pm(t,\xi),A_2^\pm(t,\xi))$ is a one-parameter families of solutions to \eqref{E:Asys} subject to the initial condition
\[
	(A_1^\pm(0,\xi),A_2^\pm(0,\xi))=(\psi^\pm_1(\xi),\psi_2^\pm(\xi)) .
\]
\end{theorem}
We remark that
the above theorem do not provide explicit forms for the asymptotic profiles of the solutions. The explicit form of the profile can be obtained by explicitly solving the ODE system \eqref{E:Asys}.
For this part, we use the argument in \cite{M}.
For the model case \eqref{E:NLS1}, it is possible to determine the asymptotic behavior of the solutions by calculating these quadratic quantities.
For $t\in \R$ and $m\in (0,1)$.
let $\pq(t,m) $ ($\mathrm{p},\mathrm{q}=\mathrm{s},\mathrm{c},\mathrm{d},\mathrm{n}$) be Jacobi elliptic function with modulus $m$. 
See \cite[Appendix A]{M3} for the definition and basic properties which we use (see e.g. \cite{BFBook} for more details).
\begin{theorem}\label{T:explicit}
In the case of 
\eqref{E:NLS1}, the solution $(A_1(\tau), A_2(\tau))$ to the ODE system
\[
	\left\{
	\begin{aligned}
	&i  \tfrac{d}{dt}A_1  = |A_2|^2 A_2,\\
	&i  \tfrac{d}{dt}A_2  = |A_1|^2 A_1,\\
	&	(A_1(0),A_2(0))= (\psi_1,\psi_2)
	\end{aligned}
	\right.
\]
is given as follows: Let $\alpha := 2 \sqrt{ |\psi_1|^4 + |\psi_2|^4 } $, $\mathcal{R}_0:= 2 \Re (\overline{\psi_1}\psi_2)$, and $m:= \frac12 - \frac{\mathcal{R}_0^2}{\alpha^2} \in [0,\frac12]$. If $\mathcal{R}_0 \neq 0 $ then
\begin{align*}
	A_1(\tau) ={}& 2^{-\frac34} \alpha^\frac12 e^{i\theta_0} \sqrt{    \dn\( \alpha \tau + t_0, m \)+ \sqrt{m} \sn\( \alpha \tau + t_0, m \)} \\
	&\times \exp\( -i \frac{\mathcal{R}_0}2 \int_0^\tau \frac{1- \sqrt{m} \sd\( \alpha \sigma + t_0, m \)}{1+ \sqrt{m} \sd\( \alpha \sigma + t_0, m \)}  d\sigma  \),\\
	A_2(\tau) ={}& 2^{-\frac14} \alpha^{-\frac12} e^{i\theta_0} \frac{\mathcal{R}_0 + i \alpha \sqrt{m} \cn (\alpha \tau + t_0,m)}{\sqrt{    \dn\( \alpha \tau + t_0, m \)+ \sqrt{m} \sn\( \alpha \tau + t_0, m \)}} \\
	&\times \exp\( -i \frac{\mathcal{R}_0}2 \int_0^\tau \frac{1- \sqrt{m} \sd\( \alpha \sigma + t_0, m \)}{1+ \sqrt{m} \sd\( \alpha \sigma + t_0, m \)}  d\sigma  \),
\end{align*}
where $\theta_0$ and $t_0$ are suitable real numbers given by initial data.
If $\mathcal{R}_0=0$ then
\begin{align*}
	A_1(\tau)&=2^{-1} \alpha^{\frac12} e^{i\theta_0}\sn ( \tfrac12 (\alpha\tau + t_0),\tfrac12) \sqrt{1+ \nd (\alpha \tau + t_0,\tfrac12)}, \\
	A_2(\tau)&=i 2^{-1} \alpha^{\frac12} e^{i\theta_0} \cd ( \tfrac12 (\alpha\tau + t_0),\tfrac12) \sqrt{1+ \nd (\alpha \tau+ t_0,\tfrac12)},
\end{align*}
where $\theta_0$ and $t_0$ are suitable real numbers given by initial data.
In particular,
\[
	| A_1(\tau)|^4 + |A_2(\tau)|^4 = \alpha^2/2,\quad
	2 \Re (\overline{A_1(\tau)}A_2(\tau)) = \mathcal{R}_0
\]
are independent of $\tau$.
\end{theorem}

\begin{remark}
The solution $(A_1, A_2)$ exhibits periodic behavior for certain values of $m$. When $m=0$ or $1/2$, the solution is periodic in $\tau$. For other values of $m$ within the range $(0,1/2)$, both components of a solution have periodic moduli with a period of $4K(m)/\alpha$, where $K(m)=\int_0^{\pi/2} (1-m \sin \theta)^{-1/2} d\theta$ represents the complete elliptic integral of the first kind. Consequently, the solution is periodic if and only if the phase factor is periodic, and the period is given by $4K(m)/\alpha$ multiplied by a rational number.
This periodicity condition can be expressed as:
\[
	  \sqrt{\tfrac12-m} \int_0^{4K(m)} \frac{1- \sqrt{m} \sd\(  \sigma , m \)}{1+ \sqrt{m} \sd\(  \sigma , m \)}  d\sigma  \in \pi \Q.
\]
\end{remark}

\subsection{Matrix-Vector representation and weak null condition}

A significant advancement in this article is the exclusion of the weak null gauge condition, which is commonly employed in the analysis of asymptotic behavior (see \cite{NST,KS,M3,U}).
The precise formulation of the condition in the current context is as follows:
\begin{enumerate}
\item[(H0)]
There exists a
positive Hermitian matrix $\mathscr{H}$ such that
\[
	\Im \(\begin{bmatrix} 
	\overline{u_1} & \overline{u_2} 
	\end{bmatrix}
	\mathscr{H}
	\begin{bmatrix}
	F_1 (u_1,u_2) \\
	F_2 (u_1,u_2)
	\end{bmatrix}
	\)
	=0
\]
holds for all $(u_1,u_2) \in \C^2$.
\end{enumerate}
One major benefit of this condition is that it guarantees the global existence of $L^2$ solutions to the NLS system \eqref{E:generalNLS}.
More specifically, if the nonlinearity satisfies this condition, the quantity
\begin{equation}\label{E:nullmass}
\frac12 \int_\R \begin{bmatrix} \overline{u_1} & \overline{u_2} \end{bmatrix} \mathscr{H} \begin{bmatrix} u_1 \ u_2 \end{bmatrix} dx
\end{equation}
becomes a conserved quantity of the corresponding NLS system \eqref{E:generalNLS}.
Since $\mathscr{H}$ is a positive Hermitian matrix, this quantity is equivalent to $\norm{u_1}{L^2}^2 + \norm{u_2}{L^2}^2$, providing a sufficient condition for the global existence of $L^2$ solutions. (See, e.g., \cite{CazBook}.)

We prove in Corollary \ref{C:H0S0} below that, as long as we consider the system \eqref{E:generalNLS}, the condition (H0) is equivalent to the following condition.
\begin{enumerate}
\item[(S0)] 
There exists a positive real-symmetric matrix $\mathscr{S}$ such that
\[
	\Im \(\begin{bmatrix} 
	\overline{u_1} & \overline{u_2} 
	\end{bmatrix}
	\mathscr{S}
	\begin{bmatrix}
	F_1 (u_1,u_2) \\
	F_2 (u_1,u_2)
	\end{bmatrix}
	\)
	=0
\]
holds for all $(u_1,u_2) \in \C^2$.
\end{enumerate}
The validity of (S0)
 is easily checked for \eqref{E:generalNLS}
by looking at the matrix-vector representation of the system as the following equivalence is valid:
\[
	\Im \(\begin{bmatrix} 
	\overline{u_1} & \overline{u_2} 
	\end{bmatrix}
	\begin{bmatrix}
	a & b \\ b & c
	\end{bmatrix}
	\begin{bmatrix}
	F_1 (u_1,u_2) \\
	F_2 (u_1,u_2)
	\end{bmatrix}
	\)
	= 0
	\Leftrightarrow 
\ltrans{(a,b,c)} \in \Ker \mathscr{A}.
\]
From this property, one easily sees that (S0) can be stated in the language of the matrix $\mathscr{A}$ as follows:
\begin{enumerate}
\item[(S1)]
$W(0,\mathscr{A}) \cap \mathcal{P}_+ \neq \emptyset$.
\end{enumerate}
The systems in the scope of this article, i.e., the systems satisfy the Assumption \ref{A:main}, do not necessarily satisfy this condition.
Indeed, the model system \eqref{E:NLS1} does not satisfy (S1).
This can be seen by \eqref{E:modelAV}: One has
$W(0;\mathscr{A}) = \Span\{ \ltrans{(0,1,0)}\}$ and hence $W(0;\mathscr{A})\cap \mathcal{P}_+=\emptyset$.
Further, a necessary condition for (S1) is $\rank \mathscr{A}\le 2$. Our model contains the system for which $\rank \mathscr{A}=3$.

Systems that satisfy the dissipative condition are also extensively studied \cite{HNS,SaSu,LS,LS2,Shimomura,Sa,LNSS1,LNSS2,LNSS5,KiSa,KiSa2}.
The condition is as follows:
\begin{enumerate}
\item[(D0)]
There exists a positive Hermitian matrix $\mathscr{H}$ such that
\[
	\Im \(\begin{bmatrix} 
	\overline{u_1} & \overline{u_2} 
	\end{bmatrix}
	\mathscr{H}
	\begin{bmatrix}
	F_1 (u_1,u_2) \\
	F_2 (u_1,u_2)
	\end{bmatrix}
	\)
	\le 0
\]
holds for all $(u_1,u_2) \in \C^2$.
\end{enumerate}
It is obvious that this condition is weaker than (H0).
This condition does not assure the existence of a conserved quantity.
However, it claims that the non-increase of the quantity \eqref{E:nullmass} as time increase.
Hence, the condition yields the global existence of solutions to \eqref{E:generalNLS}
for positive time direction.
The asymptotic behavior of solutions to NLS systems is studied 
under this condition.
Assumption \ref{A:main} does not implies (D0). Again, 
\eqref{E:NLS1} is a concrete example.


\subsection{Conservation of a quartic quantity for the ODE system}

In our model, we utilize the analysis of the ODE system \eqref{E:Asys} to derive the global existence of solutions to \eqref{E:generalNLS}. This approach is similar to those used in the previous studies where coercive conservation laws (or non-increasing property of $L^2$-norm)
are applicable. 
However, the structure in the ODE system that realizes this is different.
In our case, there exists a conserved quantity which is \emph{quartic} with respect to the unknown.
We describe how the quartic conserved quantity for the solutions of the ODE system is obtained under the assumptions of the main theorem.

To this end, we recall one result from \cite{MSU2}.
Let $(A_1,A_2) $ be a solution to the system of \eqref{E:Asys}.
Let $\mathscr{A}$ be the matrix given by \eqref{E:matrixC}.
Let
\begin{equation}\label{E:quaddef}
	\rho_1 = |A_1|^2 , \quad \rho_2 = |A_2|^2, \quad \mathcal{R} = 2\Re (\overline{A_1} A_2), \quad \mathcal{I} = 2 \Im (\overline{A_1} A_2).
\end{equation}
We have the following identity, which is one benefit of the matrix-vector representation.
\begin{proposition}[\cite{MSU2}]\label{P:quadODE}
For any $(a,b,c) \in \R^3$, one has
\begin{equation}\label{E:gQQQ1}
	\frac{d}{dt} \begin{bmatrix} \rho_1 & \mathcal{R} & \rho_2  \end{bmatrix}
	\begin{bmatrix} a \\ b \\ c \end{bmatrix}
	= \mathcal{I} \begin{bmatrix} \rho_1 & \mathcal{R} & \rho_2  \end{bmatrix}
	\mathscr{A} \begin{bmatrix} a \\ b \\ c \end{bmatrix}.
\end{equation}
\end{proposition}
Let us derive the conservation of a quartic quantity.
Suppose that $\mathscr{A}^2$ has a negative eigenvalue $-k^2$ as in Assumption \ref{A:main}.
Pick a non-zero vector $\Gamma:=\ltrans{(\gamma_1,\gamma_2 ,\gamma_3)} \in W(-k^2,\mathscr{A}^2) $
and define another eigenvector
$\widetilde{\Gamma}=\ltrans{(	\widetilde{\gamma}_1,\widetilde{\gamma}_2,\widetilde{\gamma}_3
)}$ by
\begin{equation}\label{E:tgamma}
	\widetilde{\Gamma}
	:= k^{-1}\mathscr{A} \Gamma  \in W(-k^2,\mathscr{A}^2).
\end{equation}
Then, one sees from \eqref{E:gQQQ1} that
\[
	(\gamma_1 \rho_1 + \gamma_2 \mathcal{R} + \gamma_3 \rho_2)' = k \mathcal{I} (\widetilde{\gamma}_1 \rho_1 + \widetilde{\gamma}_2 \mathcal{R} + \widetilde{\gamma}_3 \rho_2)
\]
and
\[
	(\widetilde{\gamma}_1 \rho_1 + \widetilde{\gamma}_2 \mathcal{R} + \widetilde{\gamma}_3 \rho_2)'
 = - k \mathcal{I} (\gamma_1 \rho_1 + \gamma_2 \mathcal{R} + \gamma_3 \rho_2)
\]
hold.
These two identities immediately give us the conservation of 
a quartic quantity
\begin{equation}\label{E:Q4def}
	\mathcal{Q}:=(\gamma_1 \rho_1 + \gamma_2 \mathcal{R} + \gamma_3 \rho_2)^2 + (\widetilde{\gamma}_1 \rho_1 + \widetilde{\gamma}_2 \mathcal{R} + \widetilde{\gamma}_3 \rho_2)^2.
\end{equation}
Further,
by Assumption \ref{A:main}, the vector $\Gamma$ can be chosen so that $\Gamma \in \mathcal{P}_+$.
Under this choice, one obtains
\[
	 |A_1|^2 + |A_2|^2 \sim |\gamma_1 \rho_1 + \gamma_2 \mathcal{R} + \gamma_3 \rho_2| \le \mathcal{Q}^{1/2} .
\]
Hence, the conservation of $\mathcal{Q}$ gives us the desired global bound on solutions to a system \eqref{E:Asys}. 
We use this new kind of conserved quantity to prove our main theorem.
We discuss this conserved quantity in more detail in Section \ref{S:pQ}.
It will turn out that this quantity is determined independently on the choice of $\Gamma \in W(-k^2,\mathscr{A}^2)\setminus\{0\}$ up to a positive constant. Further, $\mathcal{Q}$ is an invariant polynomial
(see Proposition \ref{P:pQ}).

%


Here, we would like to underline the center role played by the characteristic properties, such as eigenvalues and eigenspaces, of the linear map represented by matrix $\mathscr{A}$. This emphasizes that the matrix representation is more than just a square arrangement of coefficients of the system; it adeptly reflects the inherent nature of the system.
In our previous results, we focused on the fact that the rank of the matrix $\mathscr{A}$ indicates the number of (quadratic) conserved quantities for \eqref{E:generalNLS} and \eqref{E:Asys}. 
From the perspective of eigenvalues, this involves examining whether the matrix $\mathscr{A}$ possesses $0$ as an eigenvalue and, if so, determining the dimension of the eigenspace associated with the eigenvalue $0$, i.e., of $\Ker \mathscr{A}$.
%
%

\subsection{Standard forms of systems satisfying Assumption \ref{A:main}}

Let us now look for the standard form of the systems satisfying Assumption \ref{A:main}.
\begin{theorem}\label{T:std}
Let $(\lambda_1,\dots,\lambda_{12})$ be a combination of a system of the form \eqref{E:generalNLS} with unknown $(u_1,u_2)$
such that Assumption \ref{A:main} holds. Then, there exists an  invertible real $2\times 2$ matrix $\mathscr{M}$ such that if we introduce $(v_1,v_2)$ by
\[
	\begin{bmatrix} v_1 \\ v_2 \end{bmatrix}
	= \mathscr{M} \begin{bmatrix} u_1 \\ u_2 \end{bmatrix}
\]
then the system for $(v_1,v_2)$ 
is of the following form:
\begin{equation}\label{E:stdNLS}
\left\{
\begin{aligned}
	(i  \partial_t + \partial_x^2 )u_1={}&
(\eta_2 \sinh \eta_1 - \sigma \eta_3\cosh \eta_1 - \eta_2 \lambda_0) |u_1|^2 u_1 +\sinh \eta_1 (2|u_1|^2 u_2 + u_1^2\ol{u_2})\\&+ \sigma \cosh  \eta_1 |u_2|^2u_2
	 - \lambda_0 \Re (\overline{u_1} u_2) u_1 +\mathcal{V}(u_1,u_2) u_1, \\
	(i  \partial_t + \partial_x^2 )u_2={}&
\sigma \cosh \eta_1 |u_1|^2 u_1+ \sinh \eta_1 (2 u_1 |u_2|^2 + \ol{u_1}u_2^2) \\&+ ( -\sigma \eta_2 \cosh \eta_1+\eta_3 \sinh \eta_1 + \eta_3 \lambda_0 ) |u_2|^2u_2
	 + \lambda_0 \Re (\overline{u_1} u_2) u_2 + \mathcal{V}(u_1,u_2)u_2,
\end{aligned}
\right.
\end{equation}
where $\sigma\in\{\pm1\}$, and $\eta_1,\eta_2,\eta_3 \in \R$, $\lambda_0\in \R$,
and
$\mathcal{V}(u_1,u_2)=q_1 |u_1|^2 + 2q_2 \Re (\overline{u_1}u_2)  + q_3 |u_2|^2$ 
($q_1,q_2,q_3 \in \R$)
is a real-valued quadratic potential.
The matrix-vector representaion of this system is as follows:
\begin{equation}\label{E:stdAV}
	\mathscr{A}= \begin{bmatrix}
	\sinh \eta_1 & -\eta_2 \sinh \eta_1 + \sigma\eta_3\cosh \eta_1 + \eta_2 \lambda_0 & -\sigma \cosh \eta_1 \\
	0 & \lambda_0 & 0 \\
	\sigma \cosh  \eta_1& -\sigma \eta_2 \cosh \eta_1+\eta_3 \sinh \eta_1 + \eta_3 \lambda_0 & - \sinh \eta_1
	\end{bmatrix}, \quad
	\mathscr{V}= \begin{bmatrix}
	q_1 \\ q_2 \\ q_3
	\end{bmatrix}.
\end{equation}
\end{theorem}
We remark that \eqref{E:NLS1} corresponds to the case
$(\sigma,\eta_1,\eta_2,\eta_3,\lambda_0,q_1,q_2,q_3)=(1,0,0,0,0,0,0,0)$ of \eqref{E:stdNLS}.
For any choice of parameters, \eqref{E:stdNLS} satisfies Assumption \ref{A:main}.
The above form is chosen so that several mathematical objects associated with the corresponding matrix-vector representaion become clear:
\begin{itemize}
\item
The characteristic polynomial of $\mathscr{A}$ 
is $(\lambda^2+1)(\lambda- \lambda_0)$;
\item
$W(-1,\mathscr{A}^2)$ is spanned by $\ltrans{(1,0,0)}$ and $\ltrans{(0,0,1)}$, and hence $\ltrans{(1,0,1)} \in W(-1,\mathscr{A}^2) \cap \mathcal{P}_+$;
\item $W(\lambda_0,\mathscr{A})$ is spanned by $\ltrans{(\eta_2,1,\eta_3)}$.
\end{itemize}
Further, the quartic quantity for the corresponding ODE system defined by \eqref{E:Q4def} is a constant multiple of
\begin{equation}\label{E:stdQQ}
	 |A_1|^4 + 2\sigma \tanh \eta_1 |A_1|^2|A_2|^2 + |A_2|^4.
\end{equation}

The matrix $\mathscr{A}$ in \eqref{E:stdAV} satisfies
$\rank \mathscr{A} = 2$ if $\lambda_0=0$ and $\rank \mathscr{A}=3$ if $\lambda_0\neq0$.
The complete classification of the system \eqref{E:generalNLS} with $\rank \mathscr{A}=2$ is obtained in \cite{M}.
The above standard form is different from the one in \cite{M}, where the standard form is dirived by simplifying the space $W(0,\mathscr{A})$.
We disucss the comparison of
the standard forms in Appendix \ref{S:compare}.
\begin{remark}
Note that if $\lambda_0=0$
and $\eta_2\eta_3>1$ then there is a coercive mass-like conserved quantity, and hence corresponds to the case treated in \cite{LS,KS,M2}.
\end{remark}

\begin{remark}\label{R:coincidence}
The system \eqref{E:stdNLS} possesses an energy-like conserved quantity if and only if $\lambda_0=0$ and 
$
	(q_1,q_2,q_3) = q (\eta_2,1,\eta_3)
$
holds for some $q\in\R$ (see \cite[Proposition A.9]{MSU2}).
The quartic polynomial that appears as the integrand of the conserved energy-like quantity is 
\begin{multline*}
	(\eta_2^2(\sinh \eta_1 + q) + \sigma(1-\eta_2 \eta_3)\cosh \eta_1)|u_1|^4
	+4 \eta_2(\sinh \eta_1 + q) |u_1|^2\Re\overline{u_1}u_2 \\
	+2(2 \sinh \eta_1 + \eta_2\eta_3 + q) |u_1|^2|u_2|^2
	+ 2(\sinh \eta_1 + q)\Re (\overline{u_1}^2 u_2^2) \\
	+4 \eta_3 (\sinh \eta_1 + q) |u_2|^2\Re\overline{u_1}u_2 + (\eta_3^2(\sinh \eta_1 + q) + \sigma(1-\eta_2 \eta_3)\cosh \eta_1)|u_2|^4.
\end{multline*}
It is not a constant multiple of \eqref{E:stdQQ} in general.
\end{remark}

\medskip

The rest of the article is organized as follows:
Sections \ref{S:gb} and \ref{S:ab} are devoted to the proof of Theorem \ref{T:main}.
In Section \ref{S:gb}, we first summarize notations, basic facts, and properties such as the local well-posedness of \eqref{E:generalNLS}
and then we establish the global-existence-part of the main theorem.
In Section \ref{S:ab}, we complete the proof of the main result by establishing the asymptotic formula \eqref{E:mainap}.
We then prove Theorem \ref{T:explicit} in Section \ref{S:exp}.
In Section \ref{S:pre}, 
we recall
the transformation formula for the matrix-vector representation
and investigate the property of the quartic quantity $\mathcal{Q}$ in Section \ref{S:pQ}.
We then turn to the investigation of property of the systems in Section \ref{S:ps}.
We prove the equivalence of the condition (H0) and (S0). 
We finally prove
 Theorem \ref{T:std} in Section \ref{S:std}.
As an appendix, we summarize the intersection of the systems treated in this article and those previously studied.

\section{Global existence}\label{S:gb}

Now we turn to the proof of Theorem \ref{T:main}.
In this section, we establish the global-existence part of the theorem.
We also establish the global bound of the solution.

\subsection{Notations}
We introduce operators.
Let
\[
	U(t) = e^{it\partial_x^2}
\]
be the Schr\"odinger group.
The following factorization is known:
\[
	U(t) = M(t) D(t) \mathcal{F} M(t),
\]
where $\mathcal{F}$ stands for the standard Fourier transform on $\R$, $M(t) = e^{i \frac{x^2}{4t}} \times$ is a multiplication operator, 
and $D(t)$ is a dilation operator given by
\[
	(D(t)f)(x) = (2it)^{-\frac12} f(\tfrac{x}{2t}).
\]
Note that $U(-\frac1{4t}) = \mathcal{F}^{-1} M(t) \mathcal{F}$.
Let us define an operator $J(t)$ as
\[
	J(t) := x+ 2it \partial_x.
\]
We have the identity
\[
	J(t)=U(t)x U(-t)
	= M(t) D(t)i \partial_xD(t)^{-1} M(t)^{-1}= M(t) 2it \partial_x M(-t)
	,
\]
where the last two representation makes sense for $t\neq0$.

\subsection{Local well-posedness of \eqref{E:generalNLS}}
\begin{proposition}\label{P:LWP}
There exists $\widetilde{\eps}_0>0$ such that if $\eps:= \norm{u_{0,1}}_{H^{0,1}} + \norm{u_{0,2}}_{H^{0,1}}$ satisfies $\eps<\widetilde{\eps}_0$ then there exists a unique solution $(u_1,u_2) \in C([-1,1]; L^2 \times L^2)$ to \eqref{E:generalNLS} such that
$$(U(-t)u_1,U(-t)u_2) \in C([-1,1]; H^{0,1} \times H^{0,1})$$
and
\[
	\max_{t\in [-1,1]} (\norm{u_{1}(t)}_{L^2} + \norm{J(t)u_{1}(t)}_{L^2}+ \norm{u_{2}(t)}_{L^2}+\norm{J(t)u_{2}(t)}_{L^2}) \le 2\eps.
\]
\end{proposition}
The proof is standard (see \cite{CazBook}, for instance). 
We merely note that, since the nonlinearities have the gauge invariant property, the operator $J(t)$ acts like a derivative on them. More precisely, for $t\neq 0$,
\begin{align*}
	J(t) (v_1\overline{v_2}v_3)
	={}& M(t) 2it \partial_x (M(-t)
	v_1)(\overline{M(-t)v_2})(M(-t)v_3)) \\
	={}& (J(t) v_1)\overline{v_2}v_3 + v_1(\overline{J(t)v_2})v_3 + v_1\overline{v_2}(J(t)v_3).
\end{align*}

\subsection{Introduction of new variables}

Introduce a new variable
\begin{equation}\label{E:wdef}
w_j(t):=\mathcal{F} U(-t) u_j(t).
\end{equation}
Note that $U(-t) u_j(t) \in H^{0,1}$ if and only if $w_j(t) \in H^1$.
One sees that
\begin{equation}\label{E:wsys}
	i \partial_t w_j(t) 
	= 
	\tfrac1{2t} F_j(w_1,w_2) + r_j
\end{equation}
for $j=1,2$ and $t\neq0$,
where 
\begin{equation}\label{E:rdef}
r_j:=\mathrm{I}_j + \mathrm{II}_j
\end{equation}
 with
\[
	\mathrm{I}_j = (U(\tfrac1{4t})-1) D(t)^{-1} M(t)^{-1} F_j(u_1(t),u_2(t))
\]
and
\[
	\mathrm{II}_j
= \tfrac1{2t} (
F_j(U(-\tfrac1{4t}) w_1(t),U(-\tfrac1{4t}) w_2(t)) - F_j(w_1(t), w_2(t))).
\]
The following relations will be useful:
\begin{equation}\label{E:uwrelation1}
	\norm{w_j(t)}_{L^2} = \norm{u_j(t)}_{L^2},
\quad
	\norm{\partial_x w_j (t)}_{L^2} = \norm{J(t)u_j(t)}_{L^2},
\end{equation}
and
\begin{equation}\label{E:uwrelation2}
	\norm{U(-\tfrac1{4t}) w_j(t)}_{L^\infty}
		= t^{\frac12} \norm{u_j(t)}_{L^\infty}.
\end{equation}
The last identity is valid for $t\neq0$.

\subsection{Bootstrap argument}
We shall obtain a global bound
 of the small solutions.
For $\delta>0$ and $T>0$, we define
\[
	X_T:= \sup_{t\in [0,T]} (1+t)^{-\delta\eps^2}
	\sum_{j=1,2}(\norm{u_j}_{L^2} 
	+\norm{J(t) u_j(t)}_{L^2})
\]
and
\[
	Y_T:= \sup_{t\in (0,T]} t^{\frac12} \sum_{j=1,2}\norm{u_j (t)}_{L^\infty}.
\]
By the local theory (Proposition \ref{P:LWP}) and the inequality
\[
	\|u_j(t)\|_{L^\infty} = \|M(t) u_j(t)\|_{L^\infty} \lesssim 
	\|M(t) u_j(t)\|_{L^2}^\frac12 \| \partial_x M(t) u_j(t)\|_{L^2}^\frac12
		\lesssim t^{-\frac12}  (\norm{u_j}_{L^2} 
	+\norm{J(t) u_j(t)}_{L^2})
\]
for $t>0$,
for any choice of $\delta>0$, one has
\[
	X_1 + Y_1 \lesssim \eps
\]
for small $\eps>0$.
The goal of this section is to establish the following:
\begin{proposition}\label{P:gb}
There exist $\delta>0$, $C>0$, and $\eps_0\in (0,\widetilde{\eps}_0]$ such that if $\eps \in (0,\eps_0)$ then
the unique solution $(u_1(t),u_2(t))$ given in Proposition \ref{P:LWP} exists globally in time for positive time direction and obeys the bound
\begin{equation}\label{E:XYglobal}
	\sup_{T\ge1} (X_T + Y_T) \le C \eps,
\end{equation}
where $\eps$ is as in Theorem \ref{T:main}. 
\end{proposition}
\begin{remark}
We only consider the bounds for positive time, however similar bounds hold for negative time and hence the solution is global also for negative time direction.
\end{remark}

\begin{proof}
It suffices to establish the existence of $\delta$, $C_0>0$, $\widetilde{C}_0>0$, and $\eps_0>0$ such that
if $\eps \in (0,\eps_0)$ and if
\begin{equation}\label{P:gb_pf1}
	X_T  \le C_0 \eps, \quad
	Y_T \le \widetilde{C}_0 \eps
\end{equation}
for some $T\ge1$ then one has
\begin{equation}\label{P:gb_pf2}
	X_T \le \tfrac12 C_0 \eps, \quad
	Y_T \le \tfrac12 \widetilde{C}_0 \eps.
\end{equation}

Let $\delta$, $C>0$, and $\eps_0$ to be chosen later. Suppose that \eqref{P:gb_pf1}  is true for some $T\ge1$.
Let us evaluate $X_T+Y_T$.
The estimate for $X_T$ is standard.
Since the equation is cubic, one sees from the Duhamel formula that
\begin{align*}
	\sum_{j=1,2} \norm{u_j(t)}_{L^2}
	\le{}& C_1\eps + C_1\int_0^t \(\sum_{j=1,2} \norm{u_j(s)}_{L^\infty}\)^2
	\sum_{j=1,2} \norm{u_j(s)}_{L^2}ds\\
	\le{}& C_1\eps + C_1 Y_T^2 X_T
	\int_0^t (1+s)^{\delta \eps^2 -1} ds\\
	\le{}& C_1\eps + \tfrac{C_1 \widetilde{C}_0^2}{\delta} C_0\eps (1+t)^{\delta \eps^2}.
\end{align*}
We have similar estimate for $Ju_j$:
\[
	\sum_{j=1,2} \norm{J(t) u_j(t)}_{L^2}
\le C_2\eps + \tfrac{C_2 \widetilde{C}_0^2 }{\delta} C_0\eps (1+t)^{\delta \eps^2} .
\]
Hence,
\[
	X_T \le (C_1+C_2)(1 + \delta^{-1} \widetilde{C}_0^2 C_0)\eps.
\]
Now, if one has
\begin{equation}\label{P:gb_pf3}
	C_1+ C_2 \le \tfrac14 C_0
\end{equation}
and 
\begin{equation}\label{P:gb_pf4}
	4(C_1+C_2)\widetilde{C}_0^2 \le  \delta,
\end{equation}
then one obtains
\[
	X_T \le \tfrac14 C_0 \eps + \tfrac14 C_0 \eps  = \tfrac12 C_0 \eps,
\]
which is the first half of \eqref{P:gb_pf2}.

Let us turn to the $L^\infty$-decay estimate. For $t\in[0,1]$, we use
\[
	t^{\frac12}\sum_{j=1,2}\norm{u_j}_{L^\infty} 
\le Y_1.
\]
Hence, we consider $t\in [1,T]$.
It follows from Gagliardo-Nirenberg inequality that
\begin{align*}
	\norm{(U(-\tfrac1{4t})-1) w_j(t)}_{L^\infty} &\lesssim
		\norm{(U(-\tfrac1{4t})-1) w_j(t)}_{L^2}^{\frac12}
		\norm{\partial_x(U(-\tfrac1{4t})-1) w_j(t)}_{L^2}^{\frac12}\\
	&\lesssim t^{-\frac14} \norm{\partial_x w_j(t)}_{L^2}.
\end{align*}
Hence, we see from \eqref{E:uwrelation2} 
and \eqref{E:uwrelation1}
that
\begin{align*}
	t^{\frac12}\norm{u_j}_{L^\infty}=
	\norm{U(-\tfrac1{4t}) w_j(t)}_{L^\infty}
	\le{}& \norm{w_j(t)}_{L^\infty}
	+Ct^{-\frac14} \norm{\partial_x w_j(t)}_{L^2} \\
	\le{}& \norm{w_j(t)}_{L^\infty}
	+C(1+t)^{-\frac14+\delta \eps^2} X_T
\end{align*}
for $t\in[1,T]$.

Hence, we shall estimate 
$\norm{w_j(t)}_{L^\infty}$ for 
$t\in[1,T]$.
To this end, we introduce the quartic quantity.
Let $(\gamma_1,\gamma_2,\gamma_3) \in W(-k^2,A^2) \cap \mathcal{P}_+$
and $(\widetilde{\gamma}_1,\widetilde{\gamma}_2,\widetilde{\gamma}_3)$
be as in \eqref{E:tgamma}.
We define two quadratic forms
\[
	Q_1:= \gamma_1 |w_1|^2 + 2\gamma_2  \Re(\overline{w_1} w_2) + \gamma_3 |w_2|^2 = \begin{bmatrix}
	\overline{w_1} & \overline{w_2}
	\end{bmatrix}
	\begin{bmatrix}
	{\gamma}_1 & {\gamma}_2 \\ {\gamma}_2 & {\gamma}_3
	\end{bmatrix}
	\begin{bmatrix}
	w_1 \\ w_2
	\end{bmatrix}
\]
and
\[
	Q_2:= (\overline{\gamma}_1 |w_1|^2 + 2\overline{\gamma}_2  \Re(\overline{w_1} w_2) + \overline{\gamma}_3 |w_2|^2)
= 	\begin{bmatrix}
	\overline{w_1} & \overline{w_2}
	\end{bmatrix}
	\begin{bmatrix}
	\widetilde{\gamma}_1 & \widetilde{\gamma}_2 \\ \widetilde{\gamma}_2 & \widetilde{\gamma}_3
	\end{bmatrix}
	\begin{bmatrix}
	w_1 \\ w_2
	\end{bmatrix}.
\]
Further, we introduce the quartic quantity
\[
	\mathcal{Q}:= Q_1^2 + Q_2^2.
\]
Then, by means of \eqref{E:wsys},
\begin{align*}
	\partial_t Q_1
	={}& 2 \Im 
	\begin{bmatrix}
	\overline{w_1} & \overline{w_2}
	\end{bmatrix}
	\begin{bmatrix}
	\gamma_1 & \gamma_2 \\ \gamma_2 & \gamma_3
	\end{bmatrix}
	\begin{bmatrix}
	i \partial_t w_1 \\ i\partial_t w_2
	\end{bmatrix}\\
	={}& 
	2 \Im 
	\begin{bmatrix}
	\overline{w_1} & \overline{w_2}
	\end{bmatrix}
	\begin{bmatrix}
	\gamma_1 & \gamma_2 \\ \gamma_2 & \gamma_3
	\end{bmatrix}
	\begin{bmatrix}
	F_1(w_1,w_2) \\ F_2(w_1,w_2)
	\end{bmatrix}
	+R_1\\
	={}&2t^{-1}k \Im (\overline{w_1} w_2) Q_2
	+R_1,
\end{align*}
where
\[
	R_1 = 2 \Im 
	\begin{bmatrix}
	\overline{w_1} & \overline{w_2}
	\end{bmatrix}
	\begin{bmatrix}
	\gamma_1 & \gamma_2 \\ \gamma_2 & \gamma_3
	\end{bmatrix}
	\begin{bmatrix}
	r_1 \\
	r_2\\
	\end{bmatrix}
\]
and $(r_1,r_2)$ is defined in \eqref{E:rdef}.
Similarly,
\begin{align*}
	\partial_t Q_2
	={}&-2t^{-1}k \Im (\overline{w_1} w_2)
	Q_1+R_2,
\end{align*}
where
\[
	R_2 = 2 \Im 
	\begin{bmatrix}
	\overline{w_1} & \overline{w_2}
	\end{bmatrix}
	\begin{bmatrix}
	\widetilde{\gamma}_1 & \widetilde{\gamma}_2 \\ \widetilde{\gamma}_2 & \widetilde{\gamma}_3
	\end{bmatrix}
	\begin{bmatrix}
	r_1 \\
	r_2 \\
	\end{bmatrix}.
\]
One has
\[
	\partial_t \mathcal{Q} = 2R_1Q_1 + 2R_2Q_2 
	\lesssim (|w_1|+|w_2|)(|r_1|+|r_2|)
	\mathcal{Q}^\frac12 
	 \lesssim \mathcal{Q}^\frac34 
	 (|r_1|+|r_2|),
\]
which reads as
\[
	\partial_t \mathcal{Q}^\frac14	
	\lesssim |r_1|+|r_2|. 
\]
Hence, for $t\in [1,T]$,
\begin{equation}\label{E:Qbound}
	\mathcal{Q}(t)^\frac14 \le (Q_1(1)^2 + Q_2(1)^2)^\frac14
+C\int_1^t 
 	(|r_1|+|r_2|)(s)
 	 ds.
\end{equation}

Let us now estimate $r_1$ and $r_2$ for $t\in [1,T]$.
One has
\begin{equation}\label{E:I_est}
\begin{aligned}
	\norm{\mathrm{I}_j}_{L^\infty}
	\lesssim{}& 	\norm{\mathrm{I}_j}_{L^2}^\frac12
	\norm{\partial_x \mathrm{I}_j}_{L^2}^\frac12 \\
	\lesssim{}& t^{-\frac14} \norm{\partial_x D(t)^{-1} M(t)^{-1} F_j(u_1(t),u_2(t))}_{L^2} \\
	={}& t^{-\frac14} \norm{J(t) F_j(u_1(t),u_2(t))}_{L^2} \\
	\lesssim {}& t^{-\frac54} \(\sum_{j=1,2} t^{\frac12} \norm{u_j(t)}_{L^\infty} \)^2 \sum_{j=1,2}  \norm{J(t) u_j(t)}_{L^2}
\\
	\lesssim{}& t^{-\frac54+\delta \eps^2} Y_T^2 X_T
\end{aligned}
\end{equation}
and
\begin{equation}\label{E:II_est}
\begin{aligned}
	\norm{\mathrm{II}_j}_{L^\infty}
	\lesssim{}& 	t^{-1} \(\sum_{j=1,2} \norm{(U(-\tfrac1{4t})-1) w_j}_{L^\infty}\)
\(\sum_{j=1,2} \norm{U(-\tfrac1{4t}) w_j}_{L^\infty}^2 + \sum_{j=1,2} \norm{w_j}_{L^\infty}^2\)
 \\
	\lesssim{}& t^{-\frac54} \(\sum_{j=1,2} \norm{\partial_x w_j}_{L^2}\) 
\(\sum_{j=1,2} (t^{\frac12} \norm{u_j}_{L^\infty})^2 + \sum_{j=1,2} \norm{w_j}_{L^2} \norm{\partial_x w_j}_{L^2}\)
\\
	\lesssim{}& t^{-\frac54+3\delta \eps^2} (Y_T^2 X_T + X_T^3).
\end{aligned}
\end{equation}

Hence, combining \eqref{E:Qbound}, \eqref{E:I_est} and \eqref{E:II_est}, one has
\[
	\norm{\mathcal{Q}(t)^\frac14}_{L^\infty}
	\le C \eps + C \int_1^t 
	s^{-\frac54 + 3 \delta \eps^2}ds(Y_T^2 X_T + X_T^3)
 \le C\eps + C(Y_T^2 X_T + X_T^3)
\]
if $3 \delta \eps^2 \le \frac18$ for $t\in[1,T]$.
Using the equivalence relation
$\mathcal{Q}^\frac12 \sim |w_1|^2 + |w_2|^2$, one obtains $|w_1| + |w_2| \sim \mathcal{Q}^\frac14$ and hence
\[
	\sum_{j=1,2}\norm{w_j(t)}_{L^\infty} \le C \eps + C(Y_T+ X_T)  X_T^\frac12,
\]
which is the desired $L^\infty$ estimate for $\norm{w_j(t)}_{L^\infty}$.

Thus, combining above inequalities, we obtain
\[
	Y_T \le C_3 \eps + C_3(Y_T+ X_T)  X_T^\frac12 + C_3 X_T
\le C_3(1+ C_0 + (C_0+\widetilde{C}_0)C_0^{\frac12} \eps_0^{\frac12} )\eps.
\]
If we have
\begin{equation}\label{P:gb_pf5}
	4C_3(1+ C_0) \le \widetilde{C}_0
\end{equation}
and 
\begin{equation}\label{P:gb_pf6}
	C_3	(C_0+\widetilde{C}_0)C_0^{\frac12} \eps_0^{\frac12} 
\le \tfrac14 \widetilde{C}_0,
\end{equation}
then we obtain
\[
	Y_T \le \tfrac14 \widetilde{C}_0 \eps+ \tfrac14 \widetilde{C}_0 \eps =\tfrac12 \widetilde{C}_0 \eps,
\]
which is the latter half of \eqref{P:gb_pf2}.

Finally, we confirm that one can choose the parameters
$C_0$, $\widetilde{C}_0$, $\delta$, and $\eps_0$ so that 
all the argument works.
We first choose $C_0$ so that
\eqref{P:gb_pf3} holds.
We next choose $\widetilde{C}_0$ so that \eqref{P:gb_pf5} is valid.
We then choose $\delta$ and $\eps_0$ so that \eqref{P:gb_pf4} and \eqref{P:gb_pf6} hold true, respectively.
By choosing $\eps_0$ even smaller if necessary, we also assume that
$\delta \eps_0^2 \le \frac1{24}$.
Thus, we complete the proof.
\end{proof}

We conclude this section with the global $L^2\times L^2$-bound. We remark that \eqref{E:XYglobal} gives us merely
\[
	\| u_1(t)\|_{L^2}+\| u_2(t)\|_{L^2} \le C \eps (1+|t|)^{\delta \eps^2}.
\]
As mentioned in the introduction, as for systems treated in the previous studies, an effective
conservation law provides global $L^2\times L^2$-bound.
However, our model do not posses this kind of conserved quantity.
Nevertheless, we have the global bound for small solutions
by employing the conservation of $\mathcal{Q}$ for the corresponding ODE system, as in the proof $L^\infty \times L^\infty$-estimate in \eqref{E:XYglobal}.

\begin{proposition}\label{P:gb2}
Let  $\eps_0\in (0,\widetilde{\eps}_0]$ be the number given in Proposition \ref{P:gb}.
Suppose that $\eps\le \eps_0$.
There exists $C>0$ such that
the global solution $(u_1(t),u_2(t))$ given in Proposition \ref{P:gb} obeys the bound
\begin{equation}\label{E:L2global}
	\sup_{t\in \R} (\| u_1(t)\|_{L^2}+\| u_2(t)\|_{L^2})
	\le C \eps.
\end{equation}
\end{proposition}
\begin{proof}
We only consider the bound for $[1,\infty)$.
Mimicking the proof of \eqref{E:Qbound}, one sees that
\[
	\norm{\mathcal{Q}(t)^\frac14}_{L^2}
	\le \norm{(Q_1(1)^2 + Q_2(1)^2)^\frac14}_{L^2}
+C\int_1^t \sum_{j=1}^2(\norm{\mathrm{I}_j}_{L^2} + \norm{\mathrm{II}_j}_{L^2} ) ds
\]
for $t\in [1,\infty)$,
where $\mathrm{I}_j$ and $\mathrm{II}_j$ are given in \eqref{E:rdef}.
By the global bound \eqref{E:XYglobal}, one has
\begin{align*}
	\norm{\mathrm{I}_j}_{L^2}
	& \lesssim t^{-\frac12} \| \partial_x D(t)^{-1} M(t)^{-1} F_j(u_1(t),u_2(t))\|_{L^2} \\
& = t^{-\frac12} \| J(t) F_j(u_1(t),u_2(t))\|_{L^2} \\
&\lesssim t^{-\frac12} (\|u_1(t)\|_{L^\infty}+\|u_2(t)\|_{L^\infty})^2(\|J(t)u_1(t)\|_{L^2}+\|J(t)u_2(t)\|_{L^2})\\
&\lesssim \eps^3 t^{-\frac32 + \delta \eps^2}
\end{align*}
and
\begin{align*}
	\norm{\mathrm{II}_j}_{L^2}
	\lesssim{}& 	t^{-1} \(\sum_{j=1,2} \norm{(U(-\tfrac1{4t})-1) w_j}_{L^2}\)
\(\sum_{j=1,2} \norm{U(-\tfrac1{4t}) w_j}_{L^\infty}^2 + \sum_{j=1,2} \norm{w_j}_{L^\infty}^2\)
 \\
	\lesssim{}& t^{-\frac32} \(\sum_{j=1,2} \norm{\partial_x w_j}_{L^2}\) 
\(\sum_{j=1,2} (t^{\frac12} \norm{u_j}_{L^\infty})^2 + \sum_{j=1,2} \norm{w_j}_{L^2} \norm{\partial_x w_j}_{L^2}\)
\\
	\lesssim{}& \eps^3 t^{-\frac32+3\delta \eps^2} .
\end{align*}
Hence, one obtain
\[
		\sup_{t\in [1,\infty)} (\| u_1(t)\|_{L^2}+\| u_2(t)\|_{L^2})
	\lesssim \| \mathcal{Q}^{1/4}(1)\|_{L^2} + \int_0^\infty s^{-\frac32 + 3\delta \eps^2} \eps^3 ds \lesssim \eps
\]
as desired.
\end{proof}

\section{Asymptotic behavior}\label{S:ab}
In this section, we establish the asymptotics.
To begin with, we show that, thanks to the global bound, one sees that
the asymptotic profile of $u_j$ is described by
$w_j$:
\begin{proposition}\label{P:uw}
Let $\delta$ and $\eps_0$ be as in Proposition \ref{P:gb}.
It holds that
\[
	\norm{u_j(t)- (2it)^{-\frac12} e^{i \frac{x^2}{4t}} w_j(t, \tfrac{x}{2t})
	}_{L^2}
	\lesssim \eps t^{-\frac12 + \delta \eps^2}
\]
and
\[
	\norm{u_j(t)- (2it)^{-\frac12} e^{i \frac{x^2}{4t}} w_j(t, \tfrac{x}{2t})
	}_{L^\infty}
	\lesssim \eps t^{-\frac34 + \delta \eps^2}
\]
for $t>1$.
\end{proposition}
\begin{proof}
It follows that
\begin{align*}
	u_j(t) = U(t) \F^{-1} w_j(t)
	= M(t) D(t) w_j(t) + M(t) D(t) (U(-\tfrac1{4t})-1) w_j(t).
\end{align*}
For $t>0$, one has
\begin{align*}
	\norm{M(t) D(t) (U(-\tfrac1{4t})-1) w_j(t)}_{L^2} ={}& \norm{(U(-\tfrac1{4t})-1) w_j(t)}_{L^2} \\
	\lesssim{}& t^{-\frac12} 
	\norm{\partial_x w_j(t)}_{L^2} \\ 
	={}& t^{-\frac12} 
	\norm{J(t) u_j(t)}_{L^2} \\ 
	\lesssim{}& \eps t^{-\frac12 + \delta \eps^2}
\end{align*}
and
\begin{align*}
	\norm{M(t) D(t) (U(-\tfrac1{4t})-1) w_j(t)}_{L^\infty} ={}& (2t)^{-\frac12}
	\norm{(U(-\tfrac1{4t})-1) w_j(t)}_{L^\infty} \\
	\lesssim{}& t^{-\frac12} 
	\norm{(U(-\tfrac1{4t})-1)w_j(t)}_{L^2}^\frac12
	\norm{\partial_x w_j(t)}_{L^2}^\frac12 \\ 
	\lesssim{}& t^{-\frac34} 
	\norm{\partial_x w_j(t)}_{L^2} \\ 
	\lesssim{}& \eps t^{-\frac34 + \delta \eps^2}
\end{align*}
as desired.
\end{proof}
By the virtue of the proposition,
the matter boils down to finding the asymptotic behavior of $w_j$.
For this purpose, we introduce
 $(\alpha_1(t,\xi),\alpha_2(t,\xi))$ by
\begin{equation}\label{E:alphadef}
	\alpha_j(\tfrac12 \log t,\xi) = w_j (t, \xi)
\end{equation}
for $t >0$. 
One sees from \eqref{E:wsys} that $(\alpha_1,\alpha_2)$
solves
\begin{equation}\label{E:alphasys}
	i \partial_t \alpha_j
	= 
	 F_j(\alpha_1,\alpha_2) + \widetilde{r}_j
\end{equation}
for $j=1,2$, where
\begin{equation}\label{E:trdef}
\widetilde{r}_j(t;\xi) :=2e^{2t} r_j (e^{2t}, \xi).
\end{equation}
Here, $r_j$ is the error term given in \eqref{E:rdef}.
Note that \eqref{E:alphasys} is regarded as the ODE system \eqref{E:Asys}
with an error term.

Let us quickly summarize estimates on error terms $r_j$
and translate them into those on $\widetilde{r}_j$.

%
\begin{proposition}
Let $r_j$ and $\widetilde{r}_j$ be as in \eqref{E:rdef} and \eqref{E:trdef}, respectively. 
Let $\delta$ and $\eps_0$ be as in Proposition \ref{P:gb}.
If $\eps < \eps_0$ then it holds that
\[
	\norm{r_j(t)}_{L^2}
	\lesssim \eps^3 t^{-\frac32 + 3\delta \eps^2}, \quad
	\norm{r_j(t)}_{L^\infty}
	\lesssim \eps^3 t^{-\frac54 + 3\delta \eps^2}
\]
for $t>1$
and hence
\begin{equation}\label{E:trerror}
	\norm{\widetilde{r}_j(t)}_{L^2_\xi}
	\lesssim \eps^3 e^{-(1-6\delta \eps^2)t},\quad
	\norm{\widetilde{r}_j(t)}_{L^\infty_\xi}
	\lesssim \eps^3 e^{-(\frac12-6\delta \eps^2)t}
\end{equation}
for $t>-1$.
\end{proposition}
\begin{proof}
The $L^\infty$-estimate follows from \eqref{E:I_est} and \eqref{E:II_est}, together with the conclusion of Proposition \ref{P:gb}.
The $L^2$-estimate is established in the proof of Proposition \ref{P:gb2}.
The last two estimates are rephrase of the first two.
\end{proof}

Now we are in a position to complete the proof of Theorem \ref{T:main}.

\begin{proof}[Proof of Theorem \ref{T:main}]
We shall construct a one-parameter family $\{ (A_1(\cdot,\xi),A_2(\cdot,\xi)) \}_{\xi \in \R}$
of 
solutions to the ODE system \eqref{E:Asys} such that
\[
	\norm{\alpha_j(t) - A_j(t)}_{L^2_\xi} \lesssim \eps^3 e^{-(1-6\delta \eps^2 )t}, \quad
	\norm{\alpha_j(t) - A_j(t)}_{L^\infty_\xi} \lesssim \eps^3 e^{-(\frac12-6\delta \eps^2 )t}
\]
for $t\ge0$.
To this end, we use the argument in Hayashi-Li-Naumkin.
Let us construct a sequence of approximate functions $\{ (A_1^{(n)},A_2^{(n)}) \}_{n\ge0}$ as follows:
We first define $A_j^{(0)}(t,\xi) := \alpha_j(t,\xi)$ and
\[
	A_j^{(1)}(t,\xi) := \alpha_j(t,\xi) -i \int_t^\infty \widetilde{r}_j(s,\xi)ds.
\]	
We remark that, for any $t_1,t_2 \in [0,\infty)$,
\begin{equation}\label{E:appf0}
	A_j^{(1)}(t_2,\xi) - A_j^{(1)}(t_1,\xi)
	=  -i \int_{t_1}^{t_2} F_j(\alpha_1,\alpha_2)(s,\xi)ds.
\end{equation}
Then, we inductively define
\[
	A_j^{(n)}(t,\xi) := A_j^{(n-1)}(t,\xi)
	+i \int_t^\infty \(F_j(A_1^{(n-1)},A_2^{(n-1)})-F_j(A_1^{(n-2)},A_2^{(n-2)})\)(s,\xi)ds
\]
for $n\ge2$.
Let us prove that $\{ (A_1^{(n)},A_2^{(n)}) \}_{n\ge0}$ is well-defined in 
\[
	Z:=\left\{  (a_1,a_2) \in ( C([0,\infty)\times \R) )^2 ; 
	\begin{aligned}
	&\sum_{j=1,2} \norm{a_j(t)-\alpha_j(t)}_{L^2_\xi (\R)} \le C_1 \eps^3 e^{-(1-6\delta \eps^2 )t},\\
	&\sum_{j=1,2} \norm{a_j(t)-\alpha_j(t)}_{L^\infty_\xi (\R)} \le C_1 \eps^3 e^{-(\frac12-6\delta \eps^2 )t}
	\end{aligned}
	\right\}
\]
for suitable $C_1$.
Note that, since there exists $C_0$ such that
\[
	\sum_{j=1,2} \norm{\alpha_j(t)}_{L^\infty ([0,\infty)\times \R)} \le C_0 \eps
\]
in view of Proposition \ref{P:gb}, one sees that
for any choice of $C_1>0$ there exists $\eps_0 >0$
such that
\[
	\sum_{j=1,2} \norm{a_j(t)}_{L^\infty ([0,\infty)\times \R)} \le 2C_0 \eps
\]
holds
for any $(a_1,a_2) \in Z$.

By definition, $(A_1^{(0)},A_2^{(0)}) \in Z$ is obvious.
Further, thanks to \eqref{E:trerror}, there exists $C_2>0$ such that 
\[
	\sum_{j=1,2} \norm{A_j^{(1)}(t)-A_j^{(0)}(t)}_{L^p_\xi (\R)}=\sum_{j=1,2} \norm{\int_t^\infty \widetilde{r}_j(s)ds}_{L^p_\xi (\R)} \le C_2 \eps^3 e^{-(\frac12+\frac1p-6\delta \eps^2 )t}
\]
for $p=2,\infty$ and $t\ge0$.
We may suppose that $\frac12-6\delta \eps_0^2\ge \frac14$  by letting $\eps_0$ further small, if necessary.
Then, one can choose the constant from the time-integration independently of $\delta$ and $\eps$. 
We now choose $C_1=2C_2$. Then,
$(A_1^{(1)},A_2^{(1)}) \in Z$.
Now, suppose for some $k\ge1$ that
$(A_1^{(n)},A_2^{(n)}) \in Z$ and 
\begin{equation}\label{E:appf1}
	\sum_{j=1,2} \norm{A_j^{(n)}(t)-A_j^{(n-1)}(t)}_{L^p_\xi (\R)} \le 2^{-n+1}C_2 \eps^3 e^{-(\frac12+\frac1p-6\delta \eps^2 )t}	
\end{equation}
hold for $p=2,\infty$, $t\ge0$, and $n\in [1,k]$.
Then, since
\[
	|F_j(A_1^{(k)},A_2^{(k)})-F_j(A_1^{(k-1)},A_2^{(k-1)})| \le
	C\(\sum_{j=1,2}|A_j^{(k)}|+\sum_{j=1,2}|A_j^{(k-1)}|\)^2 \sum_{j=1,2} |A_j^{(k)}-A_j^{(k-1)}|  ,
\]
one sees from the assumption of the induction that $(A_1^{(k+1)},A_2^{(k+1)})$
is well-defined as an element of $C([0,\infty)\times\R)$ and
further obtains the bound
\begin{align*}
	\sum_{j=1,2} \norm{A_j^{(k+1)}(t)-A_j^{(k)}(t)}_{L^p_\xi (\R)}
\le{}& 16 C_0^2 C \eps_0^2 \int_t^\infty
2^{-k+1}C_2 \eps^3 e^{-(\frac12+\frac1{p}-6\delta \eps^2 )s}ds \\
\le {}& (128C_0^2 C \eps_0^2 )
2^{-k}C_2 \eps^3 e^{-(\frac12+\frac1p-6\delta \eps^2 )t}
\end{align*}
for $p=2,\infty$ and $t\ge0$,
which implies \eqref{E:appf1} for $n=k+1$ if $\eps_0$ is chosen so that $128C_0^2 C \eps_0^2  \le1$.
Moreover, 
\[
	\sum_{j=1,2} \norm{A_j^{(k+1)}(t)-\alpha_j(t)}_{L^p_\xi (\R)}
\le \sum_{\ell=1}^{k+1} \sum_{j=1,2} \norm{A_j^{(\ell)}(t)-A_j^{(\ell-1)}(t)}_{L^p_\xi (\R)}
\le 2C_2 \eps^3 e^{-(\frac12+\frac1p-6\delta \eps^2 )t}
\]
holds for $p=2,\infty$ and $t\ge0$,
and hence $(A_1^{(k+1)},A_2^{(k+1)}) \in Z$.
Thus, an induction argument shows that
$\{ (A_1^{(n)},A_2^{(n)}) \}_{n\ge0}\in Z$ is well-defined 
and \eqref{E:appf1} holds true for all $n\ge1$.
One easily sees that $Z$ is complete with respect to the metric on $L^\infty ([0,\infty), (L^2(\R) \cap L^\infty(\R))^2 )$. Note that \eqref{E:appf1} implies that $\{ (A_1^{(n)},A_2^{(n)}) \}_{n\ge0}\subset Z$ is a Cauchy sequence with respect to this metric.
Hence, one finds a limit $(A_1,A_2) \in Z$.

Let us confirm that the limit $(A_1,A_2)$ is a solution to \eqref{E:Asys}.
To this end, we note that \eqref{E:appf1} implies that
 $\{ (A_1^{(n)},A_2^{(n)}) \}_{n\ge0}$ converges to $(A_1,A_2) $
also in $L^1([0,\infty),  L^\infty(\R)^2)$.
Hence,
\[
	\int_t^\infty F_j(A_1^{(n)},A_2^{(n)}) (s,\xi)ds
	\to \int_t^\infty F_j(A_1,A_2) (s,\xi)ds
\]
in $L^\infty([0,\infty)\times\R)$ as $n\to\infty$.
Then, by using \eqref{E:appf0} and the definition of the sequence, it holds
for each $t_1,t_2 \in [0,\infty)$ that
\begin{align*}
	A_j(t_2)-A_j(t_1) ={}& \lim_{n\to\infty}
	( A_j^{(n)}(t_2)-A_j^{(n)}(t_1)  ) \\
	={}& A_j^{(1)}(t_2)-A_j^{(1)}(t_1) -i\lim_{n\to\infty} \int_{t_1}^{t_2}
	F_j(A^{(n-1)}_1,A^{(n-1)}_2)ds +i\int_{t_1}^{t_2}
	F_j(A^{(0)}_1,A^{(0)}_2)ds\\
	={}&-i\int_{t_1}^{t_2}
	F_j(A_1,A_2)ds.
\end{align*}
This shows that $(A_1,A_2)$ is a solution to \eqref{E:Asys}.

Finally, we show the asymptotic estimate \eqref{E:mainap}.
One sees from Proposition \ref{P:uw} and
the fact that $(A_1,A_2) \in Z$ that
\begin{align*}
	u_j(t)&= (2it)^{-\frac12} e^{i \frac{x^2}{4t}} w_j(t, \tfrac{x}{2t}) + O(\eps t^{-\frac34 + \frac1{2p} + \delta \eps^2}) \\
&=(2it)^{-\frac12} e^{i \frac{x^2}{4t}} \alpha_j(\tfrac12 \log t, \tfrac{x}{2t}) + O(\eps t^{-\frac34 + \frac1{2p}+ \delta \eps^2})\\
&=(2it)^{-\frac12} e^{i \frac{x^2}{4t}} A_j(\tfrac12 \log t, \tfrac{x}{2t}) 
+ O(t^{-\frac12+\frac1{p}} (\eps^3 t^{-\frac14-\frac1{2p}+3\delta \eps^2}))
+ O(\eps t^{-\frac34 + \frac1{2p}+ \delta \eps^2})\\
&=(2it)^{-\frac12} e^{i \frac{x^2}{4t}} A_j(\tfrac12 \log t, \tfrac{x}{2t}) 
+ O(\eps t^{-\frac34 + \frac1{2p}+ 3\delta \eps^2})
\end{align*}
in $L^p(\R)$ as $t\to\infty$ for $p=2,\infty$.
The non-endpoint case $p\in(2,\infty)$ follows by interpolation.
Writing $3\delta$ as $\delta$, we obtain \eqref{E:mainap}.
\end{proof}

\section{Explicit asymptotic profile for \eqref{E:NLS1}}\label{S:exp}

In this section, we prove Theorem \ref{T:explicit}.
We use the argument in \cite{M3}.
We first obtain the explicit formulas for the quadratic quantities of solutions.
Let us only consider the nontrivial solutions.

By \eqref{E:gQQQ1}, one has $\mathcal{R}'=0$ and hence $\mathcal{R}(t)=\mathcal{R}(0)$.
Let us consider the other quantities $(\rho_1,\rho_2,\mathcal{I})$.
We introduce $\rho_\pm(t):= \rho_1(t) \pm \rho_2(t)$.
Then, one sees that
\[
	\rho_+' =  -\mathcal{I} \rho_- ,\quad
	\rho_-' = \mathcal{I} \rho_+, \quad \mathcal{I}' = -2\rho_+\rho_-.
\]
This implies that $(f,g,h):=(\sqrt2 \rho_-, \mathcal{I}, \sqrt{2} \rho_+)$ is a solution to the ODE system
\[
	f' = gh, \quad g' = -hf, \quad h' = -fg.
\]
By \cite[Lemma 2.4]{M3}, one obtains
\begin{align*}
	\rho_+(t) &= \tfrac{\alpha}{\sqrt2} \dn\( \alpha t+t_0, \tfrac{\beta^2}{\alpha^2 } \),\\
	\rho_-(t) &= \tfrac{\beta}{\sqrt2}\sn\( \alpha t+t_0, \tfrac{\beta^2}{\alpha^2 } \),\\
	\mathcal{I}(t) &
	=\beta\cn\( \alpha t+t_0 , \tfrac{\beta^2}{\alpha^2 } \)
\end{align*}
with suitable $t_0$, where
\[
	\alpha := \sqrt{2(\rho_+(0)^2+\rho_-(0)^2)} , \quad 
	\beta := \sqrt{2\rho_-(0)^2+\mathcal{I}(0)^2}.
\]
We remark that $\alpha \ge \sqrt2 \beta \ge 0$ holds and
that $\alpha>0$ holds for all nontrivial solution.
This reads as
\begin{align*}
	\rho_1(t) &= 2^{-\frac32} \alpha \( \dn\( \alpha t + t_0, m \)+ \sqrt{m} \sn\( \alpha t + t_0, m \)\) ,\\
	\rho_2(t) &= 2^{-\frac32} \alpha \( \dn\( \alpha t + t_0, m \)-\sqrt{m} \sn\( \alpha t + t_0, m \)\) ,\\
	\mathcal{I}(t) &
	=\alpha \sqrt{m} \cn\( \alpha t + t_0, m\)
\end{align*}
with
\[
	m := \tfrac{\beta^2}{\alpha^2}=\tfrac{2(\rho_1(0)-\rho_2(0))^2+\mathcal{I}(0)^2}{4(\rho_1(0)^2+\rho_2(0)^2)} 
	= \tfrac12 - \tfrac{\mathcal{R}(0)^2}{4(\rho_1(0)^2+\rho_2(0)^2)} \in [0,\tfrac12]
\]
as desired.
Note that $t_0$ is easily specified from the relation
\[
	\tfrac1{R_1 \sqrt{m} }(\sqrt2 (\rho_1(0)-\rho_2(0)),\mathcal{I}(0))
	=(\sn(t_0, m ),\cn( t_0, m )).
\]
Also note that $ m = \frac12 $ if and only if
$\mathcal{R}(0)=0$.

Once we obtain the formula for the quadratic quantities, the formula for the solution is given by \cite[Theorem 4.1]{M3}. The theorem reads as follows in the present context:
\begin{itemize}
\item
If $\mathcal{R}(0) \neq 0$ then
\[
	A_1 (t) = \sqrt{\rho_1(t)} \tfrac{A_1(0)}{|A_1(0)|} \exp\(- i \mathcal{R}(0) \int_0^t \tfrac{\rho_2(\sigma) }{2\rho_1(\sigma)}  d\sigma \)
\]
and
\[
	A_2 (t) = \tfrac{\mathcal{R}(0)+i\mathcal{I}(t)}{2\sqrt{\rho_1(t)}}
	\tfrac{A_1(0)}{|A_1(0)|} \exp\(- i \mathcal{R}(0)\int_0^t \tfrac{\rho_2(\sigma) }{2\rho_1(\sigma)}  d\sigma \)
\]
hold for all $t\in \R$.
\item
If $\mathcal{R}(0)=0$ and
$A_1(0)\neq0$ then
\[
	A_1 (t) = (-1)^{k_1(t)} \sqrt{\rho_1(t)} \tfrac{A_1(0)}{|A_1(0)|} 
\]
and
\[
	A_2 (t) = i(-1)^{k_1(t)} \tfrac{\mathcal{I}(t)}{2\sqrt{\rho_1(t)}}
	\tfrac{\alpha_1(0)}{|\alpha_1(0)|} 
\]
hold for all $t\in I_{\max}$,
where
\[
	k_1 (t) = \begin{cases}
	\# (\{ s \in \R \ |\ \rho_1(s) =0 \} \cap [0,t]) & t>0, \\
	\# (\{ s \in \R \ |\ \rho_1(s) =0 \} \cap [t,0]) & t<0 
	\end{cases}
\]
is finite for all $t\in \R$.
\item
If $\mathcal{R}(0)=0$ and $A_2(0)\neq0$ then
\[
	A_1 (t) =  -i (-1)^{k_2(t)}
	\tfrac{\mathcal{I}(t)}{2\sqrt{\rho_2(t)}}
	\tfrac{A_2(0)}{|A_2(0)|} 
\]
and
\[
	A_2 (t) =   (-1)^{k_2(t)}
	\sqrt{\rho_2(t)}
	\tfrac{A_2(0)}{|A_2(0)|} 
\]
hold for all $t\in \R$,
where
\[
	k_2 (t) = \begin{cases}
	\# (\{ s \in \R \ |\ \rho - \mathcal{D}(s) =0 \} \cap [0,t]) & t>0, \\
	\# (\{ s \in \R \ |\ \rho - \mathcal{D}(s) =0 \} \cap [t,0]) & t<0 
	\end{cases}
\]
is finite for all $t\in \R$.
\end{itemize}

We obtain the desired representation of the solution by plugging the representation of $\rho_1$ and $\rho_2$ to these formulas.
Note that a computation shows that
a solution satisfying $\mathcal{R}=0$ has the following form:
\begin{align*}
	A_1(t)&=2^{-1} \alpha^{\frac12} e^{i\theta_0} \sn (\tfrac12( \alpha t + t_0) ,\tfrac12) \sqrt{1+ \nd ( \alpha t + t_0 ,\tfrac12)}, \\
	A_2(t)&=i 2^{-1} \alpha^{\frac12} e^{i\theta_0} \cd (\tfrac12( \alpha t + t_0),\tfrac12) \sqrt{1+ \nd ( \alpha t + t_0,\tfrac12)}
\end{align*}
with suitable  $\theta_0,t_0 \in \R$.

\section{Analysis of the quartic quantity of the ODE system}\label{S:pre}
In this section, we study the property of the conserved quantity given in \eqref{E:Q4def} of the ODE system \eqref{E:Asys}.

\subsection{Matrix-vector representation and change of variables}

The system \eqref{E:Asys} are closed under the linear transform of the unknowns.
By introducing the matrix-vector representation,
the change of the nonlinearities is described
as a matrix manipulation.
Note that it is applicable to the NLS system
\eqref{E:generalNLS}.
\begin{proposition}[\cite{MSU2}]\label{P:change}
Let $(\mathscr{A},\mathscr{V})$ be a matrix-vector representation of a system of the form \eqref{E:generalNLS} or \eqref{E:Asys}
with unknown $(u_1,u_2)$.
Define a new unknown $(v_1,v_2)$ via
\[
	\begin{bmatrix} v_1 \\ v_2 \end{bmatrix} = \mathscr{M} \begin{bmatrix} u_1 \\ u_2 \end{bmatrix},\quad
	\mathscr{M} = \begin{bmatrix} a & b \\ c & d \end{bmatrix} \in \text{GL}_2(\R).
\]
Let $(\mathscr{A}',\mathscr{V}')$ be the matrix-vectorl representation of the system for $(v_1,v_2)$. Then,
\begin{equation}\label{E:Aprime}
	\mathscr{A}' = \frac{1}{\det \mathscr{M}}  \mathscr{D}(\mathscr{M}) \mathscr{A} \mathscr{D}(\mathscr{M})^{-1}
\end{equation}
and
\begin{equation}\label{E:Vprime}
	\mathscr{V}'
	= \frac1{\det \mathscr{M}}  \mathscr{D}(\mathscr{M}) \mathscr{V}
\end{equation}
hold, where
\[
	\mathscr{D}(\mathscr{M}) := \frac1{\det \mathscr{M}}
	\begin{bmatrix} d^2 & -2cd & c^2 \\ -bd & ad+ bc & -ac \\ b^2 & -2ab & a^2 \end{bmatrix}
	\in {SL}_3 (\R),
\]
\begin{equation*}
	\mathscr{D}(\mathscr{M})^{-1} :=\frac1{\det \mathscr{M}}\begin{bmatrix} a^2 & 2ac & c^2 \\ ab & ad+ bc & cd \\ b^2 & 2bd & d^2 \end{bmatrix} \in SL_3(\R).
\end{equation*}
\end{proposition}

\subsection{Property of the Quartic quantity of the ODE system}
\label{S:pQ}
As mentioned in the introduction, we use the conservation of the quartic quantity
$\mathcal{Q}$, given in \eqref{E:Q4def}, for solutions to the ODE system \eqref{E:Asys}.
We collect properties of the quantity.
\begin{proposition}\label{P:pQ}
Suppose $\mathscr{A}^2$ has a negative eigenvalue $-k^2$ ($k>0$). 
Let $\Gamma=\ltrans{(\gamma_1,\gamma_2,\gamma_3)} \in W(-k^2,\mathscr{A}^2)$
and define a quartic homogeneous polynomial $\mathcal{Q}=\mathcal{Q}(A_1,A_2)$ by
\[
\mathcal{Q} = (\begin{bmatrix}
	\rho_1 & \mathcal{R} & \rho_2
	\end{bmatrix}\Gamma)^2
	+(\begin{bmatrix}
	\rho_1 & \mathcal{R} & \rho_2
	\end{bmatrix} \widetilde{\Gamma})^2
\]
with $\rho_1=|A_1|^2$, $\mathcal{R}= 2\Re (\overline{A_1}A_2)$, and $\rho_2=|A_2|^2$, where
$\widetilde{\Gamma}=k^{-1} \mathscr{A}\Gamma :=\ltrans{(	\widetilde{\gamma}_1 ,
	\widetilde{\gamma}_2 ,
	\widetilde{\gamma}_3
)}$.
The following three assertions are true:
\begin{enumerate}
\item
 $\mathcal{Q}$ is independent of the choice of $\Gamma \in W(-k^2,\mathscr{A}^2)$
in such a sense that another choice of non-zero $\Gamma$ gives
the same quartic polynomial up to a positive constant.
\item $\mathcal{Q}$ is invariant in the following sense: Introduce a new variable 
$(B_1,B_2)$ by
\[
	\begin{bmatrix} B_1 \\ B_2 \end{bmatrix} = 
\mathscr{M}	\begin{bmatrix} A_1 \\ A_2 \end{bmatrix}
\]
with $\mathscr{M}\in GL_2(\R)$ and let $\mathscr{A}'$ be the matrix given by \eqref{E:Aprime} with the same $\mathscr{M}$. Then, $-(k')^2$ is an eigenvalue of $(\mathscr{A}')^2$, where $k':=|\det \mathscr{M}|^{-1} k>0$.
For any non-zero $\Gamma' \in W(-(k')^2, (\mathscr{A}')^2)$, there exists $c>0$ such that
\[
	\mathcal{Q}(A_1,A_2) =c( (\begin{bmatrix}|B_1|^2 &2\Re (\overline{B_1}B_2) & |B_2|^2\end{bmatrix}\Gamma')^2+(\begin{bmatrix}|B_1|^2 &2\Re (\overline{B_1}B_2) & |B_2|^2\end{bmatrix} \widetilde{\Gamma}')^2),
\]
where $\widetilde{\Gamma}':=
\frac1{k'}\mathscr{A}'\Gamma'=\tfrac{|\det \mathscr{M}|}{k}\mathscr{A}' \Gamma'$.
\item
If $W(-k^2,\mathscr{A}) \cap \mathcal{P}_+\neq\emptyset$ then
\[
	\mathcal{Q} \sim_\Gamma (|A_1|^2 + |A_2|^2)^2
\]
holds.
On the other hand,
if $W(-k^2,\mathscr{A}) \cap \mathcal{P}_+=\emptyset$ then the equation $\mathcal{Q}=0$ possesses a nontrivial solution.
\end{enumerate}
\end{proposition}
\begin{proof}
Let us first prove that the first assertion.
To this end, pick a nonzero vector ${\bf C} \in W(-k^2,\mathscr{A}^2)$ and define
$
	\widetilde{{\bf C}}
	:= k^{-1}\mathscr{A}
	{\bf C}
\in W(-k^2,\mathscr{A}^2)$
and the corresponding quartic quantity
\[
	\tilde{\mathcal{Q}}:= (\begin{bmatrix}
	\rho_1 & \mathcal{R} & \rho_2
	\end{bmatrix}{\bf C})^2 +
(\begin{bmatrix}
	\rho_1 & \mathcal{R} & \rho_2
	\end{bmatrix}\widetilde{\bf C})^2.
\]
Since $\Gamma$ and
$\widetilde{\Gamma}$ form a basis of $W(-k^2,\mathscr{A}^2)$, there exists $\ell_1,\ell_2 \in \R$, $(\ell_1,\ell_2)\neq(0,0)$, such that
${\bf C} = \ell_1 \Gamma +  \ell_2 \widetilde{\Gamma}$.
Then, one has 
\[
\widetilde{\bf C} =
k^{-1}\mathscr{A}
	(\ell_1 \Gamma +  \ell_2 \widetilde{\Gamma})
	= \ell_2 (k^{-1}\mathscr{A}\widetilde{\Gamma}) + \ell_1
	(k^{-1}\mathscr{A}\Gamma)
= -\ell_2 \Gamma +  \ell_1 \widetilde{\Gamma}.
\]
Hence,
\begin{align*}
	\tilde{\mathcal{Q}}&= (\begin{bmatrix}
	\rho_1 & \mathcal{R} & \rho_2
	\end{bmatrix}(\ell_1 \Gamma +  \ell_2 \widetilde{\Gamma}))^2 +
(\begin{bmatrix}
	\rho_1 & \mathcal{R} & \rho_2
	\end{bmatrix}(-\ell_2 \Gamma +  \ell_1 \widetilde{\Gamma}))^2 \\
&= (\ell_1^2+\ell_2^2) ((\begin{bmatrix}
	\rho_1 & \mathcal{R} & \rho_2
	\end{bmatrix}\Gamma)^2
	+(\begin{bmatrix}
	\rho_1 & \mathcal{R} & \rho_2
	\end{bmatrix} \widetilde{\Gamma})^2)\\
&= (\ell_1^2+\ell_2^2)\mathcal{Q},
\end{align*}
which implies the desired conclusion.

We next prove the second assertion.
Pick $\mathscr{M} \in GL_2(\R)$ and let $\mathscr{A}'$ be the matrix given by \eqref{E:Aprime}.
Let $k'=|\det \mathscr{M}|^{-1} k>0$.
One immediately sees from \eqref{E:Aprime} that $- (k')^2<0$ is an eigenvalue of $(\mathscr{A}')^2$.
Further,
let
$\Gamma':=(\det \mathscr{M})^{-1}\mathscr{D}(\mathscr{M}) \Gamma$.
Then, 
\[
	(\mathscr{A}')^2 \Gamma'
	= (\det \mathscr{M})^{-3} \mathscr{D}(\mathscr{M}) \mathscr{A}^2 \Gamma
	= - (k')^2 \Gamma',
\]
which implies that $\Gamma'$ is an eigenvector of $(\mathscr{A}')^2$ associated with the eigenvalue
$- (k')^2$.
Further, one has the identity
\begin{equation}\label{E:pQpf1}
	\begin{bmatrix} |B_1|^2 & 2\Re (\overline{B_1}B_2) & |B_2|^2
	\end{bmatrix} = \det \mathscr{M}
	\begin{bmatrix} \rho_1 & \mathcal{R} & \rho_2
	\end{bmatrix}\mathscr{D}(\mathscr{M})^{-1}.
\end{equation}
This yields
\begin{equation}\label{E:pQpf2}
	\begin{bmatrix} |B_1|^2 & 2\Re (\overline{B_1}B_2) & |B_2|^2
	\end{bmatrix}\Gamma' = 
	\begin{bmatrix} \rho_1 & \mathcal{R} & \rho_2
	\end{bmatrix}\Gamma .
\end{equation}
Let us check the invariant property of $\mathcal{Q}$ for the present choice of $
\Gamma'$.
Let $\widetilde{\Gamma}':=
\frac1{k'}\mathscr{A}'\Gamma'$.
One deduces from \eqref{E:Aprime} and \eqref{E:pQpf1} that
\begin{align*}
	&\begin{bmatrix}|B_1|^2 &2\Re (\overline{B_1}B_2) & |B_2|^2\end{bmatrix} \widetilde{\Gamma}'\\
	&=(\det \mathscr{M} \begin{bmatrix} \rho_1 & \mathcal{R} & \rho_2
	\end{bmatrix} \mathscr{D}(\mathscr{M})^{-1} ) \tfrac{|\det \mathscr{M}|}{k}( (\det \mathscr{M})^{-1} \mathscr{D}(\mathscr{M})\mathscr{A}\mathscr{D}(\mathscr{M})^{-1})( (\det \mathscr{M} )^{-1}\mathscr{D}(\mathscr{M})\Gamma) \\
	&= \tfrac{\det \mathscr{M}}{|\det \mathscr{M}|} \begin{bmatrix} \rho_1 & \mathcal{R} & \rho_2
	\end{bmatrix}(k^{-1} \mathscr{A}\Gamma)\\
	&= \tfrac{\det \mathscr{M}}{|\det \mathscr{M}|} \begin{bmatrix} \rho_1 & \mathcal{R} & \rho_2
	\end{bmatrix}\widetilde{\Gamma}.
\end{align*}
Hence, together with \eqref{E:pQpf2}, one obtains
\[
	(\begin{bmatrix}|B_1|^2 &2\Re (\overline{B_1}B_2) & |B_2|^2\end{bmatrix}\Gamma')^2+(\begin{bmatrix}|B_1|^2 &2\Re (\overline{B_1}B_2) & |B_2|^2\end{bmatrix} \widetilde{\Gamma}')^2
	=\mathcal{Q}
\]
as desired. The result for other choice of $\Gamma' \in W(-( k')^2,(\mathscr{A}')^2)$ follows by combining this identity and
the result of the first assertion.

Let us proceed to the proof of the third assertion.
First suppose that $W(-k^2,\mathscr{A}^2)\cap \mathcal{P}_+ \neq \emptyset$.
Since $\mathcal{Q} \lesssim (|A_1|^2+|A_2|^2)^2$ is trivial from the fact that $\mathcal{Q}$ is a nonnegative homogeneous quartic polynomial, we prove the other direction, i.e.,
\[
	\mathcal{Q} \gtrsim (|A_1|^2+|A_2|^2)^2.
\]
Thanks to the first assertion, it suffices to show the equivalence relation for a specific choice of $\Gamma$.
We pick $\Gamma \in W(-k^2,\mathscr{A}^2) \cap \mathcal{P}_+$, which is possible because of the assumption.
Then, one has
\[
	|A_1|^2 + |A_2|^2 \lesssim
	|\begin{bmatrix}
	\rho_1 & \mathcal{R} & \rho_2
	\end{bmatrix}\Gamma| \lesssim \mathcal{Q}^{1/2}
\]
as desired.

Let us finally prove the existence of a nontrivial solution to the equation $\mathcal{Q}=0$ in the case $W(-k^2,\mathscr{A}^2) \cap \mathcal{P}_+= \emptyset$.
There are
two subscases; (a) $W(-k^2,\mathscr{A}^2) \cap \mathcal{P}_0\neq \{0\}$ and (b) $W(-k^2,\mathscr{A}^2) \cap \mathcal{P}_0= \{0\}$, where
\[
\mathcal{P}_0:=\{(a,b,c) \in \R^3 ; ac-b^2=0\}.
\]
Let us consider the subcase (a). 
Thanks to the second assertion, it suffices to prove the statement after applying suitable change of variables.
Pick nontrivial $\Gamma \in W(-k^2,\mathscr{A}^2)\cap \mathcal{P}_0$. 
Replacing $\Gamma$ with $-\Gamma$ if necessary, we suppose that $\gamma_1>0$ and $\gamma_3>0$ without loss of generality.
Note that
\[
	\begin{bmatrix}
	\rho_1 & \mathcal{R} & \rho_2
	\end{bmatrix}\Gamma
	= |\sqrt{\gamma_1 }A_1 + (\sign \gamma_2)\sqrt{\gamma_3}A_2|^2.
\]
We introduce a new variable by
\[
	\begin{bmatrix} B_1 \\ B_2 \end{bmatrix} = \begin{bmatrix} \sqrt{\gamma_1} & (\sign \gamma_2)\sqrt{\gamma_3} \\
0& 1/\sqrt{\gamma_1} \end{bmatrix}
\begin{bmatrix}
	A_1 \\ A_2
\end{bmatrix}
\]
and let $\mathscr{A}'$ be the matrix part of the system for $(B_1,B_2)$.
Then, one sees from \eqref{E:pQpf2} that
$\ltrans{(1,0,0)}\in W(-k^2,(\mathscr{A}')^2)$.
Let 
\[
	 \begin{bmatrix} \widetilde{\gamma}_1 \\ \widetilde{\gamma}_2 \\ \widetilde{\gamma}_3 \end{bmatrix}:=\mathscr{A}'\begin{bmatrix} 1 \\0\\0 \end{bmatrix} \in W(-k^2,(\mathscr{A}')^2)
	.
\]
We now claim that $\widetilde{\gamma}_3=0$. 
Indeed, otherwise
\[
	\tfrac{d}{dh}((0+h\widetilde{\gamma}_2)^2-(1+h\widetilde{\gamma}_1)(0+h\widetilde{\gamma}_3))|_{h=0} =- \widetilde{\gamma}_3 \neq0
\]
and
hence there exists $h_0$ such that
$\ltrans{(1,0,0)}+h_0\ltrans{(\widetilde{\gamma}_1,\widetilde{\gamma}_2.\widetilde{\gamma}_3)} \in W(-k^2,(\mathscr{A}')^2) \cap \mathcal{P}_+$, a contradiction.
Thus, one sees that $(B_1,B_2)=(0,1)$ is a solution to 
\[
	\mathcal{Q} \equiv |B_1|^4 + (\widetilde{\gamma}_1 |B_1|^2 +2\widetilde{\gamma}_2\Re (\overline{B_1} B_2) )^2
=0.
\]

Let us proceed to the proof in the subcase (b). 
Pick $\Gamma=\ltrans{(\gamma_1,\gamma_2,\gamma_3)} \in W(-k^2,\mathscr{A}^2)$.
By assumption, we have $\gamma_2^2-\gamma_1\gamma_3>0$.
We apply the following change of variable to make $\gamma_1=\gamma_3=0$.
Suppose $\gamma_1^2+\gamma_3^2>0$, otherwise there is nothing to be done. 
Swapping variables if necessary, one may assume that
 $\gamma_1\neq0$ without loss of generality.
Further, by replacing $\Gamma$ with $-\Gamma$ if necessary, we suppose that $\gamma_1>0$.
In this case, one has
\[
	\begin{bmatrix}
	\rho_1 & \mathcal{R} & \rho_2
	\end{bmatrix}\Gamma
	=
	\tfrac1{\gamma_1} \Re \overline{( \gamma_1 A_1 +(\gamma_2-(\gamma_2^2-\gamma_1\gamma_3)^{1/2})A_2)}( \gamma_1 A_1 +(\gamma_2+(\gamma_2^2-\gamma_1\gamma_3)^{1/2})A_2).
\]
According to the identity,
we introduce the new variable $(B_1,B_2)$  as
\[
	\begin{bmatrix} B_1 \\ B_2 \end{bmatrix} = 
	\begin{bmatrix}\gamma_1 &\gamma_2- (\sign \gamma_2)(\gamma_2^2-\gamma_1\gamma_3)^{1/2} \\
	\gamma_1 &\gamma_2+ (\sign \gamma_2) (\gamma_2^2-\gamma_1\gamma_3)^{1/2}
	\end{bmatrix}
	\begin{bmatrix} A_1 \\ A_2 \end{bmatrix}.
\]
Let $\mathscr{A}'$ be the matrix part of the system for $(B_1,B_2)$.
Then, one sees from the above  idenity and \eqref{E:pQpf2} that
$\ltrans{(0,1,0)}\in W(-(k')^2,(\mathscr{A}')^2)$,
where $k'=\gamma_1^{-1}(|\gamma_2|+(\gamma_2^2-\gamma_1\gamma_3)^{1/2})^{-1}k$.
Hence, we have obtained the desired property $\gamma_1=\gamma_3=0$.
Let 
\[
	 \begin{bmatrix} \widetilde{\gamma}_1 \\ \widetilde{\gamma}_2 \\ \widetilde{\gamma}_3 \end{bmatrix}:=\mathscr{A}'\begin{bmatrix} 0 \\1\\0 \end{bmatrix} \in W(-(k')^2,(\mathscr{A}')^2)
	.
\]
We claim that $\widetilde{\gamma}_1\widetilde{\gamma}_3<0$.
Indeed, if $\widetilde{\gamma}_1\widetilde{\gamma}_3\ge0$ then
one has
\[
	\begin{bmatrix} \widetilde{\gamma}_1\\
0\\\widetilde{\gamma}_3
	\end{bmatrix}	
	= 	\begin{bmatrix} \widetilde{\gamma}_1\\
\widetilde{\gamma}_2\\\widetilde{\gamma}_3
	\end{bmatrix}	
	-\widetilde{\gamma}_2 	\begin{bmatrix} 
	0\\1\\0
	\end{bmatrix}	
	\in W(-(k')^2,(\mathscr{A}')^2) \cap (\mathcal{P}_+\cup \mathcal{P}_0),
\]
a contradiction.
Thus, we see that $(B_1,B_2)=(|\widetilde{\gamma_3}|^{1/2},i|\widetilde{\gamma_1}|^{1/2})$ is a solution to
\[
	\mathcal{Q} \equiv (2\Re( \overline{B}_1 B_2))^2 + (\widetilde{\gamma}_1 |B_1|^2 +2\widetilde{\gamma}_2\Re \overline({B}_1 B_2) + \widetilde{\gamma}_3 |B_2|^2)^2
=0.
\]
The proof is completed.
\end{proof}

\section{Reduction to the standard form}\label{S:ps}
\subsection{The equivalence of (H0) and (S0)}\label{S:nullcond}

Here, we confirm that two condtioins (H0) and (S0) are equivalent for \eqref{E:generalNLS}.
Note that ``(S0)$\Rightarrow$(H0)'' is obvious since every real-symmetric matrix is a Hermitean matrix. We prove the opposite.
For this purpose, we introudce the following lemma. This is essentially \cite{MSU2}*{Lemma A.7}.

\begin{lemma}\label{L:quadidentity}
Let $(A_1,A_2) \in \C^2$ be a pair of complex numbers. Let
\begin{equation*}
	\rho_1 = |A_1|^2 , \quad \rho_2 = |A_2|^2, \quad \mathcal{R} = 2\Re (\overline{A_1} A_2), \quad \mathcal{I} = 2 \Im (\overline{A_1} A_2).
\end{equation*}
Then, for any $(a,b,c) \in \R^3$, one has
\begin{equation}\label{E:qqq1}
 	\Im \(\begin{bmatrix} 
	\overline{A_1} & \overline{A_2} 
	\end{bmatrix}
	\begin{bmatrix}
	a & b  \\ b  & c
	\end{bmatrix}
	\begin{bmatrix}
	F_1 (A_1,A_2) \\
	F_2 (A_1,A_2)
	\end{bmatrix}
	\)	
	= \mathcal{I} \begin{bmatrix} \rho_1 & \mathcal{R} & \rho_2  \end{bmatrix}
	\mathscr{A} \begin{bmatrix} a \\ b \\ c \end{bmatrix}.
\end{equation}
Further,
\begin{equation}\label{E:qqq2}
\begin{aligned}
 	\Im \(\begin{bmatrix} 
	\overline{A_1} & \overline{A_2} 
	\end{bmatrix}
	\begin{bmatrix}
	0 & i  \\ -i  & 0
	\end{bmatrix}
	\begin{bmatrix}
	F_1 (A_1,A_2) \\
	F_2 (A_1,A_2)
	\end{bmatrix}
	\)
= \tfrac12 \begin{bmatrix} \rho_1 & \mathcal{R} & \rho_2  \end{bmatrix} 
\mathscr{B}
\begin{bmatrix} \rho_1 \\ \mathcal{R} \\ \rho_2  \end{bmatrix} ,
\end{aligned}
\end{equation}
where $\mathscr{A}$ is given as in \eqref{E:matrixC} from $(F_1,F_2)$ and 
\begin{equation}\label{E:Bform}
	\mathscr{B} =\begin{bmatrix}
	-4\lambda_7 
	&
	\lambda_1 - \lambda_8 - \lambda_9 
	&
	2 (\lambda_2-\lambda_3-\lambda_{10}+\lambda_{11}) 
	\\
	\lambda_1 - \lambda_8 - \lambda_9 
	& 
	2(\lambda_3 - \lambda_{11})
	&
	\lambda_4+\lambda_5 - \lambda_{12} 
	\\
	2 (\lambda_2-\lambda_3-\lambda_{10}+\lambda_{11}) 
	& 
	\lambda_4+\lambda_5 - \lambda_{12} 
	&
	4 \lambda_6
	\end{bmatrix}.
\end{equation}
\end{lemma}
\begin{remark}
The matrix $\mathscr{B}$ depends only on the matrix part $\mathscr{A}=\{a_{ij}\}_{1 \le i,j, \le 3}$ of the matrix-vector representation.
Indeed, it is written as
\[
	\mathscr{B} =\begin{bmatrix}
	4 a_{13}
	&
	-a_{12}+2a_{23} 
	&
	2a_{11}+2a_{33}  
	\\
	-a_{12}+2a_{23} 
	& 
	-2 a_{22} 
	&
	-a_{32} + 2a_{21} 
	\\
	2a_{11}+2a_{33} 
	& 
	-a_{32} + 2a_{21} 
	&
	4 a_{31} 
	\end{bmatrix}.
\]
We remark that there is no contribution from the vector part $\mathscr{V}$ in the left hand side of \eqref{E:qqq1} and \eqref{E:qqq2}.
\end{remark}

\begin{corollary}\label{C:H0S0}
Let $(F_1,F_2)$ be as in
\eqref{E:nonlinearity}.
Suppose that there exists
a Hermitian matrix
\[
	\mathscr{H} = \begin{bmatrix}
	p & q_1 + i q_2 \\ q_1 - i q_2 & r
	\end{bmatrix}, \quad p,q_1,q_2,r \in \R
\]
such that
\[
	\Im \(\begin{bmatrix} 
	\overline{A_1} & \overline{A_2} 
	\end{bmatrix}
	\mathscr{H}
	\begin{bmatrix}
	F_1 (A_1,A_2) \\
	F_2 (A_1,A_2)
	\end{bmatrix}
	\)
	=0
\]
holds for all $(A_1,A_2) \in \C^2$.
Then, it holds that
\[
	\Im \(\begin{bmatrix} 
	\overline{A_1} & \overline{A_2} 
	\end{bmatrix}
	\begin{bmatrix}
	p & q_1  \\ q_1  & r
	\end{bmatrix}
	\begin{bmatrix}
	F_1 (A_1,A_2) \\
	F_2 (A_1,A_2)
	\end{bmatrix}
	\)
	=0
\]
for all $(A_1,A_2) \in \C^2$.
In particular, condition (H0) implies
(S0).
\end{corollary}
\begin{proof}
We suppose that $q_2\neq0$, otherwise the conclusion is obvious.
Pick $(A_1,A_2) \in \C^2$ and define $\rho_1$, $\rho_2$, $\mathcal{R}$, and $\mathcal{I}$ as in \eqref{E:quaddef}.
Utilizing \eqref{E:qqq1} and \eqref{E:qqq2}, one has
\[
	0=	\Im \(\begin{bmatrix} 
	\overline{A_1} & \overline{A_2} 
	\end{bmatrix}
	\mathscr{H}
	\begin{bmatrix}
	F_1 (A_1,A_2) \\
	F_2 (A_1,A_2)
	\end{bmatrix}
	\)
	= \mathcal{I} \begin{bmatrix} \rho_1 & \mathcal{R} & \rho_2  \end{bmatrix}
	\mathscr{A} \begin{bmatrix} p \\ q_1 \\ r \end{bmatrix}
	+\tfrac{q_2}2 \begin{bmatrix} \rho_1 & \mathcal{R} & \rho_2  \end{bmatrix} 
\mathscr{B}
\begin{bmatrix} \rho_1 \\ \mathcal{R} \\ \rho_2  \end{bmatrix}.
\]

We will prove that $\mathscr{B}=0$. Then, we obtain the desired conclusion from \eqref{E:qqq1}.
Let $b_{ij}$ is the $(i,j)$-entry of $\mathscr{B}$.
We substitute $(A_1,A_2)=(1,\tau)$ with $\tau\in \R$. Then, we obtain an idenity between two polynomials in $\tau$.
Since
$(\rho_1,\mathcal{R},\mathcal{I},\rho_2)=(1,2\tau,0,\tau^2)$, one sees  from the comparison of the coefficients that $b_{11} = b_{12}  = b_{23} = b_{33}=0$
and $b_{13} + 2b_{22} =0$.
Next, we do a similar test with the choice
$(A_1,A_2)=(1,i\tau)$ with $\tau\in \R$.
Then, noting that $(\rho_1,\mathcal{R},\mathcal{I},\rho_2)=(1,0,2\tau,\tau^2)$, one further obtains $b_{13}=0$ from the comparison of the two polynomials. Thus, one has $\mathscr{B}=0$.
\end{proof}

At the end of this section, let us observe that \eqref{E:NLS1} does not satisfies the condtion (D0).
Recalling its matrix-vector representaion and \eqref{E:qqq1} and \eqref{E:qqq2}, one has
\[
		\Im \(\begin{bmatrix} 
	\overline{A_1} & \overline{A_2} 
	\end{bmatrix}
	\begin{bmatrix}
	p & q_1 + i q_2 \\ q_1 - i q_2 & r
	\end{bmatrix}
	\begin{bmatrix}
	F_1 (A_1,A_2) \\
	F_2 (A_1,A_2)
	\end{bmatrix}
	\) = -r \rho_1 \mathcal{I} + p \rho_2 \mathcal{I} -2 q_2 \rho_1^2 + 2 q_2 \rho_2^2.
\]
Suppose that this quantity is nonpositive for all $(A_1,A_2)\in \C^2$, for some $p,q_1,q_2,r \in \R$ with $p,r>0$ and $pr>q_1^2+q_2^2$.
One first obtain $q_2=0$ by substition of $A_1=0$ or $A_2=0$. 
Then, it is easy to see that this quantity cannot be nonpositive for all $(A_1,A_2)\in \C^2$. This can be verified, for instance,
by testing the value when $(A_1,A_2)=(1,\pm i \tau)$
with small or large $\tau\in\R$.

\subsection{Reduction to the standard form}\label{S:std}
In this section, we discuss the procedure of the reduction.

\begin{proof}[Proof of Theorem \ref{T:std}]
Suppose $W(-k^2,\mathscr{A}^2) \cap \mathcal{P}_+ \neq \emptyset$.
We apply several changes of variables to reduce $\mathscr{A}$ to the desired form.

First, we introduce the variable $(u_1^1,u_2^1):=(\sqrt{k} u_1, \sqrt{k} u_2 )$. Then the corresponding matrix  becomes
$\mathscr{A}_1=k^{-1} \mathscr{A}$, thanks to \eqref{E:Aprime}.
Hence, 
$(\mathscr{A}_1)^2$ has the eigenvalue $-1$.
One sees that $W(-1,(\mathscr{A}_1)^2)$
has an intersection with $\mathcal{P}_0$.
Hence, swapping the variables of the system if necessary,
there exists $\beta \in \R$ such that
$\ltrans{(1,\beta,\beta^2)} \in W(-1,(\mathscr{A}_1)^2)$.
We introduce a new variable $(u_1^2,u^2_2)$ by
\[
	\begin{bmatrix} u_1^2 \\ u_2^2 \end{bmatrix}
	= \begin{bmatrix} 1 & \beta \\ 0 & 1 \end{bmatrix}
	\begin{bmatrix} u_1^1 \\ u_2^1 \end{bmatrix}.
\]
Then, one sees from \eqref{E:Aprime} that
the matrix $\mathscr{A}_2$ corresponding to the transformed system
satisfies
$\ltrans{(1,0,0)} \in W(-1,(\mathscr{A}_2)^2)$.
Let 
\[
	 \begin{bmatrix} \widetilde{\gamma}_1 \\ \widetilde{\gamma}_2 \\ \widetilde{\gamma}_3 \end{bmatrix}:=\mathscr{A}_2\begin{bmatrix} 1 \\0\\0 \end{bmatrix} \in W(-1,(\mathscr{A}_2)^2)
	.
\]
Note that $W(-1,(\mathscr{A}_2)^2)$ is spanned by
$\ltrans{(1,0,0)}$ and $\ltrans{(\widetilde{\gamma}_1,\widetilde{\gamma}_2,\widetilde{\gamma}_3)}$.
One then sees that $\widetilde{\gamma}_3\neq0$
since otherwise $W(-1,(\mathscr{A}_2)^2)$ does not have an intersection with $\mathcal{P}_+$. According to the identity
\[
	\widetilde{\gamma}_1 |u_1^2|^2 + 2 \widetilde{\gamma}_2 \Re (\overline{u_1^2}u_2^2) + \widetilde{\gamma}_3 |u_2^2|^2 =
r^{-2} (\widetilde{\gamma}_1-\tfrac{\widetilde{\gamma}_2^2}{\widetilde{\gamma}_3})|r u_1^2|^2 +r^2 \widetilde{\gamma}_3 |  \tfrac{\widetilde{\gamma}_2}{\widetilde{\gamma}_3r} u_1^2+r^{-1} u_2^2|^2,
\]
where $r>0$ is a number to be chosen later,
we introduce a new variable $(u_1^3,u^3_2)$ by
\[
	\begin{bmatrix} u_1^3 \\ u_2^3 \end{bmatrix}
	= \begin{bmatrix} r & 0 \\ \frac{\widetilde{\gamma}_2}{\widetilde{\gamma}_3r} & r^{-1} \end{bmatrix}
	\begin{bmatrix} u_1^2 \\ u_2^2 \end{bmatrix}
\]
and let $\mathscr{A}_3$ be the matrix corresponding to the system for $(u_1^3,u^3_2)$.
It holds from
\[
	(\mathscr{A}_3)^n = \begin{bmatrix} r^{-2} & -2\frac{\widetilde{\gamma}_2}{\widetilde{\gamma}_3r^2}  & \frac{\widetilde{\gamma}_2^2}{\widetilde{\gamma}_3^2r^2} \\ 0 & 1 & -\frac{\widetilde{\gamma}_2}{\widetilde{\gamma}_3} \\ 0 &0 & r^2 \end{bmatrix} (\mathscr{A}_2)^n 
\begin{bmatrix} r^2 & 2\frac{\widetilde{\gamma}_2}{\widetilde{\gamma}_3} & \frac{\widetilde{\gamma}_2^2}{\widetilde{\gamma}_3^2r^2} \\ 0 & 1 & \frac{\widetilde{\gamma}_2}{\widetilde{\gamma}_3r^2} \\ 0 & 0 & r^{-2} \end{bmatrix}
\]
for $n=1,2$ that
\[
	(\mathscr{A}_3)^2\begin{bmatrix} 1 \\0\\0 \end{bmatrix}
	=- \begin{bmatrix} 1 \\0\\0 \end{bmatrix},\qquad
	\mathscr{A}_3\begin{bmatrix} 1 \\0\\0 \end{bmatrix}
	= \begin{bmatrix} \widetilde{\gamma}_1-\tfrac{\widetilde{\gamma}_2^2}{\widetilde{\gamma}_3} \\ 0 \\r^4 \widetilde{\gamma}_3 \end{bmatrix}.
\]
We introduce $\eta_1 = \sinh^{-1} (\widetilde{\gamma}_1-\tfrac{\widetilde{\gamma}_2^2}{\widetilde{\gamma}_3})\in \R$ and choose
$r$ so that
	$r^4\widetilde{\gamma}_3 =  \sigma \cosh \eta_1$,
where $\sigma\in\{\pm1\}$.
Let $\lambda_0$ be the real eigenvalue of $\mathscr{A}_3$ and let $\eta_2,\eta_3 \in \R$ be the number such that
$\ltrans{(\eta_2,1,\eta_3)}\in W(\lambda_0,\mathscr{A}_3)$. Note that we can pick such a vector since $W(-1,(\mathscr{A}_3)^2)=\{(x,0,z) \in \R^3 ; x,z\in\R\}=(\ltrans{(0,1,0)})^\perp$ and $W(-1,(\mathscr{A}_3)^2) \oplus W(\lambda_0,\mathscr{A}_3) = \R^3$. 
Thus,
\[
	\begin{bmatrix}
	1 &  \sinh \eta_1 & \eta_2 \\
	0 & 0 & 1 \\
	0 &  \sigma \cosh \eta_1 & \eta_3
	\end{bmatrix}^{-1}
	\mathscr{A}_3
	\begin{bmatrix}
	1 & \sinh \eta_1 & \eta_2 \\
	0 & 0 & 1 \\
	0 & \sigma \cosh \eta_1 & \eta_3
	\end{bmatrix}
	=
	\begin{bmatrix}
	0 & -1 & 0 \\
	1 & 0 & 0 \\
	0 & 0 & \lambda_0
	\end{bmatrix},
\]
which implies
\begin{equation}\label{E:stdAform}
	\mathscr{A}_3 =\begin{bmatrix}
	\sinh \eta_1 & -\eta_2 \sinh \eta_1 + \sigma \eta_3\cosh \eta_1 + \eta_2 \lambda_0 & -\sigma \cosh \eta_1 \\
	0 & \lambda_0 & 0 \\
	\sigma \cosh  \eta_1& -\eta_2 \sigma \cosh \eta_1+\eta_3 \sinh \eta_1 + \eta_3 \lambda_0 & - \sinh \eta_1
	\end{bmatrix}.
\end{equation}
Thus, we get a system of the desired form \eqref{E:stdAV}.

For a given pair $(\mathscr{A}=(a_{ij})_{1\le i,j \le 3}, \mathscr{V}=(q_k)_{1\le k \le 3} )\in M_3(\R) \times \R^3$, the system is constructed with the following formula (see \cite{MSU1}).
\begin{equation}\label{E:CVNLS}
\left\{
\begin{aligned}
	(i  \partial_t + \partial_x^2 )u_1={}&
-(a_{12}+a_{23}) |u_1|^2 u_1 +a_{11}(2|u_1|^2 u_2 + u_1^2\ol{u_2})+a_{21} (2 u_1 |u_2|^2 + \ol{u_1}u_2^2) \\&+ a_{31} |u_2|^2u_2
	 - (\tr{\mathscr{A}}) \Re (\overline{u_1} u_2) u_1 +\mathcal{V}(u_1,u_2) u_1, \\
	(i  \partial_t + \partial_x^2 )u_2={}&
-a_{13} |u_1|^2 u_1-a_{23}(2|u_1|^2 u_2 + u_1^2\ol{u_2})-a_{33} (2 u_1 |u_2|^2 + \ol{u_1}u_2^2) \\&+ (a_{21}+a_{32}) |u_2|^2u_2
	 + (\tr{\mathscr{A}}) \Re (\overline{u_1} u_2) u_2 + \mathcal{V}(u_1,u_2)u_2,
\end{aligned}
\right.
\end{equation}
where $\mathcal{V}(u_1,u_2)=q_1 |u_1|^2 + 2q_2 \Re (\overline{u_1}u_2)  + q_3 |u_2|^2$ is a real-valued quadratic potential.
Plugging \eqref{E:stdAform} to \eqref{E:CVNLS}, we obtain the system \eqref{E:stdNLS}.
\end{proof}

\appendix

\section{Comparison of the standard forms}\label{S:compare}

As mentioned in the introduction, the standard form \eqref{E:stdNLS} belongs to the class treated in \cite{M} when $\lambda_0=0$.
In this section, we see the correspondence between the standard form given there.
The argument of the proof of Theorem \ref{T:std} provides a concrete procedure of obtaining a system of the form \eqref{E:stdNLS} from a system satisfying Assumption \ref{A:main}.
Similarly, the argument in \cite{M} also give us a procedure to reduce a system into a standard form in \cite{M}.
Hence, here we concentrate on characterizing the standard form in \cite{M} which satisfies Assumption \ref{A:main}.
Note that the other direction, i.e., the characterization of systems of the form \eqref{E:stdNLS} which is treated in \cite{M} is simple as mentioned just above; $\lambda_0=0$.
To avoid inessential complexity, we concentrate on the matrix-part of the systems. 
\subsection{The case $\lambda_0=0$ and $\eta_2\eta_3>1$}
In this case, we have
$W(0,\mathscr{A}) \cap \mathcal{P}_+ \neq \emptyset$.
Hence, a system of the form \eqref{E:stdNLS} is reduced to an elliptic system in the terminology of \cite{M}.

Recall that the matrix-part of the standard form of an elliptic system is as follows (see \cite[Theorem 1.10]{M} and \cite{M3}):
\begin{equation}\label{E:estdMatform}
	\mathscr{A}=\begin{bmatrix}
	p_1+p_5 & -2p_2-2p_3-2p_4 & - p_1- p_5 \\
	p_{2}-p_3 & 2p_1& - p_{2}+p_3 \\
	-p_1+p_5 & 2p_{2}+2p_3-2p_4 & p_1 - p_5
	\end{bmatrix},
\end{equation}
where
\[
	(p_1,p_2,p_3,p_4,p_5) \in \cup_{1 \le j \le 4} \mathcal{T}_j
\]
with
\begin{align*}
	\mathcal{T}_1 &:= \{ (x_1,x_2,x_3,0,0) \in S^4\ ;\ x_1\ge0, x_3\ge0, x_3 \neq 2^{-1/2} \},\\
	\mathcal{T}_2 &:= \{ (x_1,x_2,0,x_4,0) \in S^4\ ;\ x_1\ge0, 0< x_4<1 \},\\
	\mathcal{T}_3 &:= \{ (x_1,x_2,x_3,x_4,x_5) \in S^4\ ;\ x_1\ge0, x_3>0, x_4\ge0, x_5\ge0, x_4^2+x_5^2>0\}\\
	&\qquad \setminus  \{ ( \sqrt{\tfrac{1-r^2}2}\sin 2\zeta, \sqrt{\tfrac{1-r^2}2}\cos 2\zeta, \sqrt{\tfrac{1-r^2}2} ,r \cos \zeta, r\sin \zeta) \in S^4; r \in (0,1) ,\zeta \in [0,\pi/2]\}, \\
	\mathcal{T}_4 &:= \{ (x_1,x_2,x_3,x_4,x_5) \in S^4\ ;\ x_1\ge0, x_2>0,x_3>0,x_4<0,x_5>0 \}.
\end{align*}

By a direct computation, one sees that $\mathscr{A}$ given in \eqref{E:estdMatform} satisfies $W(-k^2, \mathscr{A}^2) \neq \{0\}$ for some $k>0$ if and only if
\begin{equation}\label{E:Econd1}
	p_1 = 0 ,\quad p_2^2 >p_3^2.
\end{equation}
Note that $k= 2\sqrt{p_2^2-p_3^2}$ in this case.
Further, under this assumption, $W(-4(p_2^2-p^2_3), \mathscr{A}^2) \cap \mathcal{P}_+ \neq \emptyset $ if and only if
\begin{equation}\label{E:Econd2}
	(\tfrac{p_5}{p_2-p_3})^2 + 	(\tfrac{p_4}{p_2+p_3})^2 >1.
\end{equation}

Thus, one sees that
the combination $(p_1,p_2,p_3,p_4,p_5)$
satisfies Assumption \ref{A:main} if and only 
if \eqref{E:Econd1} and \eqref{E:Econd2} are valid, and such a combination
exists in $\mathcal{T}_2$, $\mathcal{T}_3$, and $\mathcal{T}_4$.
Note that the system given by a combination in
\[
	\{ (0,x_2,0,x_4,0) \in S^4\ ;\ x_4>|x_2|>0 \}
\subset \mathcal{T}_2
\]
belongs to Case 8 of \cite{M3}, in which case we obtain explicit formula for solution to the corresponding ODE system.

\subsection{The case $\lambda_0=0$ and $\eta_2\eta_3=1$}
In this case, a system of the form \eqref{E:stdNLS} is reduced to a parabolic system in the terminology of \cite{M}.
Thanks to \cite[Theorem 1.11]{M},
the parabolic system which possesses a pure imaginary eigenvalue exists only in the set
$\mathcal{T}_{\mathrm{p}.8}$.
Hence, we suppose that
the matrix-part of a system in this class is written as
\begin{equation*}
\mathscr{A}=	\begin{bmatrix}
	0 & a_{12} & a_{13} \\
	0 & 0 & \sigma_2 \\
	0 & \sigma_1 & 0
	\end{bmatrix}
\end{equation*}
with
$a_{12} \in \R$, $a_{13} \ge0$, and $\sigma_1,\sigma_2 \in \{\pm 1\}$.
One verifies that a matrix $\mathscr{A}$ of this form satisfies $W(-k^2,\mathscr{A}^2) \neq \{0\}$ for some $k>0$ if and only if
$\sigma_1 \sigma_2 = -1$.
Note that $k=1$ in this case. Further, under this assumption,
$W(-1, \mathscr{A}^2) \cap \mathcal{P}_+ \neq \emptyset $ holds true if and only if
$a_{13}^2 - 4 \sigma_2 a_{12} > 0$.

Thus, we see that the following set gives us the complete list of the standard forms of parabolic systems satisfying Assumption \ref{A:main}:
\[
	\left\{
	\begin{bmatrix}
	0 & a_{12} & a_{13} \\
	0 & 0 & \sigma \\
	0 & -\sigma & 0
	\end{bmatrix} ; a_{12},a_{13} \in \R, \sigma \in \{\pm 1\}
,a_{13}^2 - 4 \sigma_2 a_{12} >0
	\right\} \subset \mathcal{T}_{\mathrm{p}.8}.
\]

\subsection{The case $\lambda_0=0$ and $\eta_2\eta_3<1$}
In this case, a system of the form \eqref{E:stdNLS} is reduced to a hyperbolic system in the terminology of \cite{M}.
Thanks to \cite[Theorem 1.12]{M},
the parabolic system which possesses a pure imaginary eigenvalue exists only in $\mathcal{T}_{\mathrm{h}.6}$.
Hence, we consider the matrix $\mathscr{A}$  of the form
\begin{equation*}
	\mathscr{A}=\begin{bmatrix}
	a_{11} & 0 & -1 \\
	a_{21} & 0 & a_{23} \\
	1 & 0 & a_{33}
	\end{bmatrix}
\end{equation*}
with
$[a_{11}>-a_{33}] \vee [[a_{11}=-a_{33}]\wedge [a_{21} \ge -a_{23}]]$
and $[a_{11}a_{33}\neq-1]\vee [a_{11}a_{23}\neq -a_{21}]$.
By a simple calculation,
a matrix $\mathscr{A} \in 	\mathcal{T}_{\mathrm{h}.6}$ satisfies $W(-k^2,\mathscr{A}^2) \neq \{0\}$ for some $k>0$ if and only if
\[
	a_{11}=-a_{33} \in (-1,1) \quad \text{and} \quad a_{21} \ge -a_{23}.
\]
Note that $k=\sqrt{1-a_{11}^2}$ in this case.
Further, under this assumption,
$W(-1+a_{11}^2, \mathscr{A}^2) \cap \mathcal{P}_+ \neq \emptyset $ holds true if and only if
\[
	(1-a_{11}^2)^2+ 4 (a_{21}a_{11}+a_{23})(a_{23}a_{11}+a_{21}) > 0.
\]

Thus, we see that the following set gives us the complete list of the standard forms of parabolic systems satisfying Assumption \ref{A:main}:
\[
	\left\{
\begin{bmatrix}
	a_{11} & 0 & -1 \\
	a_{21} & 0 & a_{23} \\
	1 & 0 & -a_{11}
	\end{bmatrix} ; 
	\begin{aligned}
	&a_{11} \in (-1,1),a_{21},a_{23} \in \R, a_{21}+ a_{23} \ge 0,\\
	&(1-a_{11}^2)^2+ 4 (a_{21}a_{11}+a_{23})(a_{23}a_{11}+a_{21}) > 0
	\end{aligned}
	\right\} \subset \mathcal{T}_{\mathrm{h}.6}.
\]

\subsection*{Acknowledgements} 
The author was supported by JSPS KAKENHI Grant Numbers JP23K20803, JP23K20805, and JP24K00529.

\end{document}